\documentclass[12pt]{amsart}
\usepackage{amssymb,amsmath,amsfonts,latexsym}
\usepackage{mathtools}
\usepackage{tikz-cd}
\usepackage[colorlinks=true, linkcolor=blue, citecolor=red, anchorcolor=orange,]{hyperref}

\newcommand{\newsection}[1]{\setcounter{equation}{0} \section{#1}}
\def\textmatrix#1&#2\\#3&#4\\{\bigl({#1 \atop #3}\ {#2 \atop #4}\bigr)}
\def\dispmatrix#1&#2\\#3&#4\\{\left({#1 \atop #3}\ {#2 \atop #4}\right)}
\newcommand{\be}{\begin{equation}}
	\newcommand{\ee}{\end{equation}}
\newcommand{\ben}{\begin{eqnarray*}}
	\newcommand{\een}{\end{eqnarray*}}

\newcommand{\bi}{\begin{itemize}}
	\newcommand{\ei}{\end{itemize}}

\newtheorem{thmI}{Theorem}

\newtheorem{theorem}{Theorem}[section]
\newtheorem{lemma}[theorem]{Lemma}
\newtheorem{proposition}[theorem]{Proposition}
\newtheorem{corollary}[theorem]{Corollary}
\theoremstyle{definition}
\newtheorem{definition}[theorem]{Definition}
\newtheorem{example}[theorem]{Example}

\newtheorem{question}[theorem]{Question}
\newtheorem{remark}[theorem]{Remark}
\numberwithin{equation}{section}

\begin{document}

\title{On $R$-equivalence of Automorphism groups of Associative algebras}

\author[Das]{Dibyendu Das}
\address{Indian Statistical Institute, Statistics and Mathematics Unit, 8th Mile, Mysore Road, Bangalore, 560059,
India}
\email{ddas4282@gmail.com}


	
\begin{abstract}
Let $A$ be a finite-dimensional associative algebra with unity over a field $k$. The primary aim of this paper is to study the rationality properties of the group of all $k$-algebra automorphisms of $A$, as an affine algebraic group over an arbitrary field $k$. We investigate mainly the $R$-equivalence property of $G_A$ over a perfect field $k$, where $G_A$ is the identity component of $\mathrm{Aut}_{k}(A)$. In this paper, we prove that $G_A$ is $R$-trivial under some sufficient conditions on $A$. Additionally, we construct a class of algebras to show the existence of groups of the form $G_A$ that are not $R$-trivial.
\end{abstract}
	
\maketitle

\tableofcontents

\newsection{Introduction and Main results }\label{sec: intro}

Let $k$ be an arbitrary field and $A$ be a finite-dimensional unital associative algebra over $k$. We denote by $\text{Aut}_k(A)$ the group of all $k$-algebra automorphisms of $A$.
The understanding of the algebraic group $\text{Aut}_k(A)$ and its different properties is of deep interest. Many researchers have looked into the group of automorphisms of finite-dimensional unital associative algebras from an algebraic perspective. For example, see the papers (\cite{MM}, \cite{Djokovic}, \cite{Luks}, \cite{H.G}, \cite{Stanley}) and their references for more details. The papers \cite{Frolich} and \cite{Yau2} discuss relations of $\text{Aut}_k(A)$ with the Picard group of $A$ and the classification of isolated hypersurface singularities, respectively. The article (\cite{Pollack}) by David R. Pollack served as our main inspiration, where he has shown how different properties of $G_A$ (the identity component of $\text{Aut}_k(A)$) affect the properties of $A$ over an algebraically closed field. The solvability and unipotency of the group $\text{Aut}_k(A)$ are important in this context, as is apparent from Yau's work in (\cite{Yau2}). Perepechko (\cite{Perepechko}) and Saor\'in et al. (\cite{AS1}, \cite{AS2}) have made significant contributions to the understanding of the solvability of $\text{Aut}_k(A)$. Moreover, Saor\'in and Guil Asensio explored the Picard groups along with the outer automorphism groups of finite-dimensional associative algebras in their papers (\cite{AS2}, \cite{AS3}). Nonetheless, little is known about the geometric properties of $\text{Aut}_k(A)$. This paper studies some birational properties of $\text{Aut}_k(A)$ as a $k$-group. One such property is similar to the concept of path connectedness. In algebraic geometry, this notion corresponds to the $R$-equivalence of a variety. We explore the rationality and $R$-equivalence properties of $G_A$ (the identity component of $\mathrm{Aut}_k(A)$), which is crucial in the study of $\text{Aut}_k(A)$.

\vspace{0.2in}
\noindent
For a connected algebraic group $G$ defined over an arbitrary field $k$, an important question is the rationality problem: Is the function field $k(G)$ purely transcendental over $k$? Determining whether a variety is rational is a central problem in algebraic geometry. In this context, the notion of $R$-equivalence was introduced by Y.~Manin around $1972$ (see \cite{Manin}, Chapter $2$). The knowledge that a group is $R$-trivial helps in studying its rationality properties. A systematic study of this notion for linear algebraic groups was then initiated by J.-L.~Colliot-Thélène and J.-J.~Sansuc \cite{CTS}. Around $1996$, A.~Merkurjev \cite{AM2} and Ph.~Gille \cite{Gille2} studied $R$-triviality for the group of proper projective similitudes of a quadratic form or, more generally, automorphism groups of central simple associative algebras with involutions. For more on $R$-equivalence of automorphism groups of central simple algebras with involutions, one can look at  \cite{AP}, \cite{Preeti} and their references. On the other hand, the automorphism groups of general finite-dimensional associative algebras are not as thoroughly investigated in this context.
\vskip2mm 
\begin{question}
Let $G=G_A$, where $A$ is a finite-dimensional associative algebra over a perfect field $k$. Is the function field $k(G_A)$ purely transcendental over $k$?  
\end{question}
\vskip2mm
To the best of our knowledge, no results address the above question. $R$-triviality is a necessary condition for a group to be rational (see Section \ref{R-equivalence}). This article provides a systematic framework for investigating the $R$-equivalence property of connected algebraic groups that are of the form $G_A$, where $A$ is a split finite-dimensional associative algebra (see Definition \ref{Definition 2.6}), which is essential for the investigation of any finite-dimensional associative algebra (see Section \ref{sec: Preliminaries} and Section \ref{R-equivalence}). Our results are given for the case of split finite-dimensional associative algebras so that, after Morita reduction, we can apply the quiver and relation representation to analyse the $R$-triviality. More precisely, we have the following results for $G_A$ over a perfect field $k$:

\vspace{0.1in}
\begin{thmI}[]
$(=\text{Theorem }\ref{Theorem 4.18})$
Let $A$ be a split finite-dimensional associative algebra over a perfect field $k$ with $A_s$ a semisimple subalgebra of $A$ and $J$ is the Jacobson radical of $A$ such that $A= A_s\oplus J$ (vector space decomposition), and $G_{A, A_s}\coloneqq\{\sigma\in G_A: \sigma(a)=a \quad \forall a\in A_s\}$. Then, $G_A$ is $R$-trivial if and only if $G_{A,A_s}$ is $R$-trivial. Suppose any one of the following holds:
\begin{enumerate}
    \item $J^2=0$;
    \vskip2mm
    \item $\emph{dim}(J/J^2)\leq 3$;
    \vskip2mm
    \item $\emph{rank} (G_{A,A_s})=\emph{dim}(J/J^2)$.
\end{enumerate}
\vskip2mm
 Then, $G_A=\mathrm{Aut}_k(A)^{0}$ is $R$-trivial.
\end{thmI}
\vskip2mm
\begin{thmI}[]
$(=\text{Theorem }\ref{Theorem 4.30})$ Let $A$ be a split finite-dimensional associative algebra over a perfect field $k$. Assume that $\emph{\text{Aut}}_k(A)$ is a nilpotent group. Then, $G_A=\emph{\text{Aut}}_k(A)^{0}$ is $k$-rational.
\end{thmI}

\vspace{0.1in}
\noindent
At first, we reduce the problem of studying $R$-equivalence of $G_A$, where $A$ is a split associative algebra, to that of studying $R$-equivalence of $G_A$, where $A$ is a split basic associative algebra (see Definition \ref{Definition 2.6} and Section \ref{R-equivalence}).
\vskip2mm 
\noindent
We prove in Theorem \ref{Theorem 4.18} that $G_A$ is $R$-trivial if and only if $G_{A,A_s}$ is $R$-trivial, where $A_s$ is a semisimple subalgebra of $A$ such that $A=A_s\oplus J$ (vector space decomposition) with Jacobson radical $J$ and $G_{A,A_s}\coloneqq\{\sigma\in G_A: \sigma(a)=a \quad \forall a\in A_s\}$. In Proposition \ref{Proposition 4.31}, we show that the property of $R$-triviality of $G_A$ is invariant under \emph{Morita equivalence} (see Definition \ref{Definition 2.3}). Consequently, it suffices to examine split basic associative algebras. Thus, we can use the quiver and relation representation of split basic algebras to study the $R$-triviality of $G_A$ (see Section \ref{sec: commutative algebra}). Furthermore, Proposition \ref{Proposition 4.33} indicates that for a split local associative algebra $A$, the group $G_A$ is stably birational to $\text{ Im}(\Phi_A)$, where $\Phi_A: G_A\rightarrow \mathrm{GL}(J/J^2)$ is the canonical map and $J$ is the Jacobson radical of $A$. In particular, we have \[G_A(F)/R\cong \mathrm{Im}(\Phi_A)(F)/R\] for every field extension $F/k$ and for every split local algebra $A$ over $k$ with $\text{char }k=0$ (see Section \ref{R-equivalence} for notation). Therefore, the analysis can be effectively narrowed down to the study of $\text{ Im}(\Phi_A)$.
\vskip4mm
\noindent
Thereafter, we are primarily concerned with the study of $R$-triviality of the automorphism groups of split local commutative algebras. We use the \emph{quiver and relation} representation of $A$ to examine the $R$-equivalence property of $G_A$, where $A$ is a split local commutative algebra. Under certain new sufficient conditions without any restrictions on $\text{dim}(J/J^2)$, we prove that $G_A$ is $R$-trivial, thereby establishing that the conditions of Theorem \ref{Theorem 4.18} are sufficient but not necessary for commutative algebras (see Theorem \ref{Theorem 5.12} and Theorem \ref{Theorem 5.20}). Also, we show the existence of a local commutative algebra $A$ with $\text{dim}(J/J^2)\geq 6$ and $G_A$ is not $R$-trivial (see Theorem \ref{Theorem 6.3}). More precisely, we have the following results stated below for $G_A$ over a perfect field $k$, where $A$ is a split local commutative algebra:
\vspace{0.1in}
\begin{thmI}[]
$(=\text{Theorem }\ref{Theorem 5.12})$ Let $A$ be a local commutative algebra over a perfect field $k$ with $\emph{dim}(J/J^2)=n$. Suppose $A\cong k[X_1,\dots, X_n]/I$ by quiver and relation representation of $A$, where $I=\langle X_1,\dots, X_n\rangle^l+\langle f\rangle$, $f$ is a homogeneous polynomial and $2\leq \mathrm{ deg}(f)<l$ is the Loewy length of $A$. Then, $G_A$ is $R$-trivial as a $k$-group if and only if $\mathrm{Sim}^{0}(f)$ is $R$-trivial.
\end{thmI}
\vskip2mm
\begin{thmI}[]
$(=\text{Theorem }\ref{Theorem 5.20})$ Let $A$ be a split local commutative graded by the radical algebra over a perfect field $k$ with $\emph{dim}(J/J^2)=n$. Let $W$ be the minimal degree subspace associated to $A$ (see Remark \ref{Remark 5.19}). Let any of the conditions listed below be satisfied:
\begin{enumerate}
 \item $\emph{char }k=0$ and $W$ contains a non-singular homogeneous polynomial of degree at least $3$ with $\emph{\text{dim }}W\leq 3$;
\vskip2mm
\item $\emph{char }k=p>3$ and $W$ contains a non-singular homogeneous polynomial of degree $d$ with $2<d<p$ and $\emph{\text{dim }}W\leq 3$.
\end{enumerate}
 Then $G_A$ is $k$-rational.
\end{thmI}
\vskip2mm
\begin{thmI}[]
$(=\text{Theorem }\ref{Theorem 6.3})$ For each $n\geq1$, there exists a finite-dimensional commutative algebra $A_n$ with $\emph{dim} \displaystyle (\frac{J_n}{J_n^2})\geq 6$, where $J_n$ is the Jacobson radical of $A_n$ over a field $k_n$, such that $G_{A_n}$ is not $R$-trivial as a $k_n$-group.  
\end{thmI}

\vskip2mm
\noindent
Clearly $\text{Aut}_{k}(A)$ is an affine algebraic group. Associative algebras and affine algebraic groups have a rich structure theory. From now onwards, algebra means a finite-dimensional unital associative algebra. Both (\cite{Humphreys}) and (\cite{Springer}) serve as general references for algebraic groups, and we use the theory of associative algebras from (\cite{Curtis}, \cite{Pierce}). The concept of representation of a split basic finite-dimensional associative algebra by \emph{quiver and relation} is taken from (\cite{ISD}, \cite{AS1}, \cite{PG}). For the theory of Galois cohomology and quadratic forms, one can follow \cite{KMRT}. In this article, we will deal with finite-dimensional associative algebras over perfect fields (some of the results may be valid over arbitrary fields, but we mainly stick to perfect fields). Unless specified otherwise, there are no other restrictions on the field $k$.
\vskip2mm

The rest of the paper is organized as follows: In Section \ref{sec: Preliminaries}, we recall some well-known results on the structure of finite-dimensional associative algebras and provide some basic definitions that are used throughout the paper. Section \ref{Review} reviews Pollack's results over perfect fields. Section \ref{R-equivalence} is devoted to the study of $R$-equivalence of the automorphism groups of split finite-dimensional associative algebras. In Section \ref{sec: commutative algebra}, we establish the results about the $R$-equivalence of automorphism groups of split local commutative algebras by using quiver and relations. Finally, in Section \ref{counter}, we provide examples that demonstrate the existence of groups of the form $G_A$ which are not $R$-trivial, where $A$ is a finite-dimensional commutative algebra.
 
\section{Preliminaries}\label{sec: Preliminaries}

This section discusses the structure theory of finite-dimensional associative algebras and provides some basic definitions that are necessary for subsequent sections.
\vskip2mm
\noindent
Throughout this paper, $A$ is a finite-dimensional associative algebra over a field $k$ with unity, and $J$ is its Jacobson radical (in short, we call it the radical). By an affine algebraic group over a perfect field $k$ (such as $\mathrm{Aut}_k(A)$ and its closed subgroups), we mean a smooth linear algebraic group over $k$. For any affine algebraic group $G$ over $k$, $G^{0}$ denotes its identity component. In this article, $\mathrm{Aut}_k(A)^{0}=G_A$. 
Let $G$ be an algebraic group over $k$. Then $G(F)$ denotes its $F$-rational points, where $F$ is a field extension of $k$ and by $G_K$ we denote the base extension of the group $G$ over any field extension $K$ of $k$. Here, in this paper we use $G_{A_F}$ to denote ${\mathrm{Aut}_k(A)}^{0}_F$, where $A_F=A\otimes_{k} F$ and $F$ is a field extension of $k$.
\vskip2mm
\noindent
Next, we recall the principal structure theorem for finite-dimensional associative algebras: the Wedderburn-Malcev decomposition.

\vspace{2mm}

\begin{theorem}\label{Theorem 2.1}
\textbf{$($\emph{Wedderburn-Malcev}}, \cite{Curtis}, \emph{Theorem} $72.19)$ Let $A$ be a finite-dimensional associative algebra over $k$ (arbitrary field) with radical $J$ such that the residue class algebra $B=A/J$ over $k$ is separable. Then, there exists a semisimple subalgebra $A_s$ of $A$ such that $A=A_s\oplus J$ (vector space decomposition). If there exists a semisimple subalgebra $A'_s$ such that $A=A'_s\oplus J$, then there exists an element $r\in J$ such that $A'_s=(1+r)A_s(1+r)^{-1}$.
\end{theorem}

\vskip2mm

\begin{corollary}\label{Corollary 2.2}
$($\cite{Pierce}, \emph{Page} $211)$ If $k$ is a perfect field, and $A$ is finite-dimensional associative $k$-algebra, then there is a semisimple subalgebra $A_s$ of $A$ such that $A/J\cong A_s$ and $A=A_s\oplus J$. Moreover, $A_s$ is unique upto
conjugation by units of the form $(1+r)$, where $r\in J$.
\end{corollary}

\vskip2mm 
\noindent
Let $A^{\times}$ denote the group of units in $A$ and $k^{\times}=k-\{0\}$. We have an obvious map from $A^{\times}$ to $\text{Aut}_k(A)$ given  by $\psi:A^{\times}\rightarrow \text{Aut}_k(A) \text{ with }\psi(a)=\text{Int}(a)$, where $\text{Int}(a)$ is the automorphism  $x\mapsto axa^{-1}$.

\vskip2mm
\noindent
This is a morphism of algebraic groups, whose image is $\text{Inn}(A$), the group of all inner automorphisms of $A$. Observe that $\{1+r:r\in J\}$ is a closed, connected, unipotent subgroup of $A^{\times}$ whose image $\hat{J}\coloneq\{\text{Int}(1+r):r\in J\}$ is normal in $\text{Aut}_k(A)$ and $\hat{J}\subset U$, where $U$ is the $k$-unipotent radical of $G_A$.

\vskip2mm

\begin{definition}\label{Definition 2.3}$($\cite{Pierce}, Section $9.6$$)$
Two algebras $A$ and $B$ are said to be \emph{Morita equivalent} if their left module categories are equivalent.
\end{definition}

\vskip2mm
\begin{definition}\label{Definition 2.4}
An algebra $A$ over $k$ is a \emph{basic algebra} if $A/J\cong \bigoplus_{i=1}^{n} D_i$, where each $D_i$ is a division algebra over $k$ and $n\in \mathbb{N}$. In particular, if $A/J\cong D$ for some division algebra $D$ over $k$, we call $A$ a \emph{local algebra} over $k$.  
\end{definition}

\vskip2mm 
\begin{theorem}\label{Theorem 2.5}
$($\cite{Pierce}, \emph{Section} $9.6)$ For each finite-dimensional associative algebra $A$ over a perfect field $k$, there always exists a basic algebra $B$ over $k$ which is Morita equivalent to $A$.   
\end{theorem}
\vskip2mm
\noindent
Thus, over a perfect field $k$, every finite-dimensional associative algebra \( A \) can be decomposed as \( A = A_s \oplus J \) (vector space decomposition), as established by Corollary \ref{Corollary 2.2}. In this decomposition, \( A_s\cong A/J \) is a semisimple subalgebra that is isomorphic to \( \bigoplus_{i=1}^{m} M_{n_{i}}(D_{i}) \), where each $D_i$ is a division algebra over $k$ for all $1\leq i\leq m$ and \( J \) denotes the Jacobson radical of the algebra \( A \). 
\vskip2mm
\begin{definition}\label{Definition 2.6}
An associative algebra $A$ over $k$ is called a \emph{split associative algebra} if $A/J\cong \bigoplus_{i=1}^{m} M_{n_{i}}(k)$, where $m\in \mathbb{N}$. An associative algebra $A$ over $k$ is called a \emph{split basic associative algebra} if $A/J\cong \bigoplus_{i=1}^{m}k$, and an associative algebra $A$ over a field $k$ is a \emph{split local associative algebra} if $A/J\cong k$.  
\end{definition}

\vskip2mm
\section{Review of Pollack's Results over Perfect fields}\label{Review}
\vskip2mm

In this section, we discuss the results of (\cite{Pollack}) related to the reductivity and semisimplicity of $G_A$ in terms of the properties of $A$ over a perfect field $k$, where $A$ is a split associative algebra over $k$. Over an algebraically closed field, any finite-dimensional associative algebra is split. Therefore, for a split finite-dimensional associative algebra over a perfect field $k$, the proof of Theorem \ref{Theorem 3.1} follows from a similar line of argument as done by Pollack in \cite {Pollack}. The other proofs of this section are given for the sake of completeness. Let us recall a few notations, $\text{Inn}(A)$ is the inner automorphism group of $A$, $\hat{J}=\{\text{Int}(1+r):r\in J\}$, $U$ denotes the $k$-unipotent radical of $G_A$ and $A_{s}^{\times}=\text{Units of } A_s$.

\vskip2mm
\begin{theorem}\label{Theorem 3.1}
Let $A$ be a split finite-dimensional associative algebra over a perfect field $k$, and let $J$ be the radical of $A$. If $U$ is the $k$-unipotent radical of $G_A$, then $J^2=0$ if and only if $J\subset A^{U}$, where $A^U\coloneq \{a\in A:\sigma(a)=a \quad \forall \sigma\in U(F), \text{ for all field extension $F$ over $k$ } \}$. 
\end{theorem}
\vskip1mm
\begin{proof}
Over a perfect field $k$, a group is reductive if and only if the $k$-unipotent radical is trivial. We also have $A/J\cong \bigoplus_{i=1}^{m}M_{n_{i}}(k)$, $U_{\bar{k}}=R_u(G_A)_{\bar{k}}$ and $(A_{\bar{k}})^{U_{\bar{k}}}=A^{U}\otimes \bar{k}$. Therefore, the proof of the above statement follows in a similar line of argument as done by Pollack in Theorem $1.1$ of \cite{Pollack}.
\end{proof}
\vskip2mm

\vspace{2mm}
\begin{corollary}\label{Corollary 3.2}
Let $A$ be a split finite-dimensional associative algebra over a perfect field $k$. Then $G_A$ is reductive if and only if $J^2=0$ and $J\subset Z(A)$, where $Z(A)$ is the center of $A$.
\end{corollary}
\vskip1mm
\begin{proof}
    Let $G_A$ be reductive, then $U=\{Id\}$, $U$ is the $k$-unipotent radical. Hence, it follows that $A^U=A\supset J$. Therefore, by Theorem \ref{Theorem 3.1}, we have $J^2=0$. Since $U=\{Id\}$, we also have $\hat{J}=\{Id\}$, which further implies $J\subset Z(A)$. Conversely, if we assume $J^2=0$, then it follows from the proof of Theorem \ref{Theorem 3.1} (or Theorem $1.1$ of \cite{Pollack}) that the $k$-unipotent radical of $G_A$ is $\hat{J}=U$. However, $J\subset Z(A)$ shows that $\hat{J}=\{Id\}$. Hence, the group $G_A$ is reductive.
\end{proof}

\vskip2mm
\begin{corollary}\label{Corollary 3.3}
Let $A$ be a split finite-dimensional associative algebra over a perfect field $k$. Then $G_A$ is semisimple if and only if $A\cong \bigoplus_i M_{n_i}(k)$. 
\end{corollary}
\vskip1mm
\begin{proof}
$G_A$ is semisimple implies that it is reductive, so $J^2=0$ and $J\subset Z(A)$. We will show that $J=0$. Suppose $J\neq 0$ and $G_A$ is semisimple, then we have $\sigma\in G_A$ implies $\sigma(A_s)=A_s$ as $\hat{J}=\{Id\}$ and $\sigma(J)=J$ (see Lemma \ref{Lemma 4.19}). Now, consider $T_{\lambda}$ for each $\lambda\in \mathbb{G}_m$, which is given by $T_{\lambda}(a+r)=a+\lambda r$. Using the fact that $J^2=0$ and $J\subset Z(A)$, we can prove that $T_{\lambda}\in G_A$ and $T_{\lambda}\in Z(G_A)$ for each $\lambda\in \mathbb{G}_m$, which is a contradiction to semisimplicity of $G_A$ as $\{T_\lambda:\lambda\in \mathbb{G}_m\}\cong \mathbb{G}_m$ embeds in $Z(G_A)$. Hence, $J$ must be trivial. Conversely, if $A\cong \bigoplus_i M_{n_i}(k)$, then $G_{A}\cong \prod_i \mathrm{PGL_{n_i}}$, which is a semisimple group.
\end{proof}

\vskip2mm
\begin{theorem}\label{Theorem 3.4}
 Let $T$ be a torus of rank $r$ over a perfect field $k$ such that there exists a split finite-dimensional associative algebra $A$ over $k$ with $T\cong \emph{\text{Aut}}_k(A)$, then $T$ is a split torus.
\end{theorem}
\begin{proof}
It is given that $T\cong \text{Aut}_k(A)=G_A$ over a perfect field $k$, where $A$ is a split finite-dimensional algebra over $k$. Since $T$ is nilpotent, we have $\mathrm{Inn}_{A}(A_{s}^{\times})=\{\mathrm{Int}(a)\in G_A :a\in A_{s}^{\times}\}$ is nilpotent, which implies $A_{s}^{\times}$ is also nilpotent, where $A_s\cong \bigoplus_{i}M_{n_{i}}(k)$. So we get that $A_s\cong \bigoplus_i k$ and $A$ must be a split basic algebra. Since $G_A$ is reductive, we have $J\subset Z(A)$ by Corollary \ref{Corollary 3.2}. Therefore, $A$ is commutative, and we have $A\cong \bigoplus_i A_i$, where each $A_i$ is a split local algebra (see \cite{ISD}). Hence we have $G_A=\prod_i G_{A_i}\cong T$. So each $G_{A_i}$ is either a torus or a trivial group. If $G_{A_i}$ is a trivial group, then $A_i=k$, and if $G_{A_i}$ is a torus, then again by using Corollary \ref{Corollary 3.2}, we get that $J_i^2=0$ and $J_i\subset Z(A_i)=k\oplus Z(J_i)$, where $J_i$ is the Jacobson radical of $A_i$. However, by the quiver and relation representation of $A_i$, we get that $A_i\cong k[X_1,\dots,X_n]/I$ with $\langle X_1,\dots,X_n\rangle^l\subseteq I \subseteq\langle X_1,\dots,X_n\rangle^2$, where $l=\text{min }\{m:J_i^m=0\}\geq 2$ and $n\geq 1$ (see \cite{AS1} or Section \ref{sec: commutative algebra}). Since $J_{i}^2=0$, we have $A_i\cong k[X_1,\dots,X_n]/\langle X_1,\dots, X_n\rangle^2$. Therefore, for each $G_{A_i}$ to be a torus, we must have $A_i\cong k[X]/\langle X\rangle^2$ by Theorem $1.3$ of \cite{AS1}. Finally we have either $A_i\cong k$ or $A_i\cong k[X]/\langle X\rangle^2$. Hence, $A\cong (\bigoplus_j k)\oplus (\bigoplus_i k[X]/\langle X\rangle^2).$ Now it follows that $G_A\cong \mathbb{G}_{m,k}^r$.
\end{proof}

\vskip2mm

\begin{remark}\label{Remark 3.5}
   Let $A$ be a split finite-dimensional algebra over the perfect field $k$, with $J$ being the radical of $A$. If $A=A_s\oplus J$, where $A_s\cong \bigoplus_{i=1}^{m} M_{n_i}(k)$, then $A\otimes_k K=(A_s\otimes_k K)\oplus (J\otimes_k K)$, where $A_s\otimes_k K\cong \bigoplus_{i=1}^{m} M_{n_i}(K)$, for every extension $K$ of $k$ and $m\in \mathbb{N}$.
\end{remark}

\vskip2mm
\section{$R$-Equivalence and $R$-triviality}\label{R-equivalence}
\vskip2mm
Let $k$ be a perfect field and $G_A$ be the identity component of the automorphism group of a split finite-dimensional associative algebra $A$. In this section, we discuss the $R$-equivalence of $G_A$.

\vskip2mm
\noindent
Let $X$ and $Y$ be varieties (irreducible) over $k$. We write $X\cong_b Y$ if $X$ and $Y$ are birationally equivalent, that is, the function fields $k(X)$ and $k(Y)$ are isomorphic over $k$, and we use the notation $X \cong^{s.b} Y$ if $X$ and $Y$ are stably birational, that means, $X\times \mathbb{A}_k^m\cong_b Y\times \mathbb{A}_k^n$ for some $m,n$.

\vskip2mm
\noindent
We say that a variety $X$ is $k$-\emph{rational} (resp. $k$-\emph{stably rational)} if $X\cong_b\mathbb{A}_k^n$ (resp. $X\times_k \mathbb{A}_k^n\cong_b \mathbb{A}_k^{m})$  for some $n,m\geq 0$.
  
\vskip2mm
\noindent
Let $X$ be an irreducible variety over a field $k$ with $X(k)\neq\emptyset$. We define two points $x, y \in X(k)$ to be $R$-equivalent ($x \sim y$) if there exists a sequence of points $x_0= x, x_1,...,x_n=y$ in $X(k)$ and rational
maps $f_i:\mathbb{A}^{1}_{k}\dashrightarrow X$, $1\leq i\leq n$, defined over $k$, regular at $0$ and $1$, such
that $f_{i}(0)=x_{i-1},f_{i}(1)=x_{i}$ (see \cite{Manin}).

\vskip2mm
\noindent
Let $G$ be a connected algebraic group defined over $k$. The set of points in $G(k)$ that are $R$-equivalent to identity in $G(k)$ forms a normal subgroup of $G(k)$, denoted by $RG(k)$. The set $G(k)/R$ of $R$-equivalence classes in $G(k)$ is in canonical bijection with $G(k)/RG(k)$ and thus has a natural group structure. We identify $G(k)/R$ with the group $G(k)/RG(k)$. In the literature, it is known that this group is very useful in studying the rationality properties of $G$, for example, one may look at the references (\cite{Gille1}, \cite{Voskrenskii}).

\vskip2mm
\noindent
We call $G$ is \emph{$R$-trivial}, if $G(L)/R = \{1\}$ for all field extensions $L$ of $k$. It is a fact that if $G$ is $k$-rational, then $G$ is $R$-trivial (see
\cite{Voskrenskii}, Chap. $6$, Prop. $2$). Also, because tori of rank at most $2$ are rational, it follows that any reductive algebraic group of rank at most $2$ is rational (\cite{CM}, Cor. $2.3$). If $X$ and $Y$ are two irreducible varieties over $k$ then,
\begin{enumerate}
    \item $(X\times Y)(L)/R\cong X(L)/R\times Y(L)/R $, where $L/k$ is a field extension.
    \vskip2mm
    \item $R_{F/k}(X)(k)/R=X(F)/R$, where $F/k$ is a finite extension (\cite{Voskrenskii}).
\end{enumerate}

\vskip2mm
\noindent
If a group $G$ is $k$-split, then $G$ is $k$-rational (see \cite{VT}, \cite{CP}). It is also known that any connected unipotent algebraic group over a perfect field is rational (\cite{Voskrenskii}).

\vskip2mm
\begin{proposition}\label{Proposition 4.1}
Let $A$ be a finite-dimensional associative $k$-algebra. Assume that $\emph{\text{Aut}}_k(A)= \emph{Inn}(A)$. Then $\emph{\text{Aut}}_k(A)$ is $R$-trivial.
\end{proposition}

\vskip1mm
\begin{proof}
Consider the following rational map $\theta :\mathbb{A}^1_{k}\dashrightarrow \text{Inn}(A)$ defined by 
\begin{equation*}\label{eq4}
\theta(t)(x)=\{(1-t).1+ta\}x\{(1-t).1+ta\}^{-1}
\end{equation*}
\vskip1mm
\noindent
Therefore, $\theta(0)=Id$ and $\theta(1)=\text{Int}(a)$. It is clear that $L_a:A\rightarrow A$ is invertible if and only if $a$ is invertible, where $L_a(x)=ax$. 
It follows that $c=\{(1-t).1+t.a\}$ is invertible except for a finite set. Therefore, the map $\theta$ is well-defined. Hence, $k$-rational points are $R$-equivalent. Now, $\text{Aut}_{k}(A)(F)\cong \text{Aut}_F(A\otimes F)=\text{Inn}(A\otimes F)$ for any field extension $F$ over $k$. Again, by a similar line of argument, we can show that $F$-rational points are $R$-equivalent. Hence, we are done. 
\end{proof}

\vskip1mm

\begin{example}\label{Example 4.2}
Let $A=\bigoplus_{i=1}^{m} M_{n_i}(k)$, then by the classical result of Skolem and Noether $G_A\cong \prod_{i=1}^{m} \text{PGL}_{n_{i}}$. However, $\text{PGL}_n$ is $R$-trivial as a $k$-group. Hence, $G_A$ is $R$-trivial.
\end{example}

\vskip2mm

\begin{example}\label{Example 4.3}
Let $A$ be the set of upper triangular matrices in $M_n(k)$. Here, the radical $J=$ upper triangular nilpotent matrices and $A_s=$ diagonal matrices, and we have $A=A_s\oplus J$ as the Wedderburn-Malcev decomposition. Then $\text{Aut}_k(A)=\text{Inn}(A)$ as shown in (\cite{Pollack}, \cite{Stanley}). Therefore, this group is $R$-trivial by Proposition \ref{Proposition 4.1}.
\end{example}

\vskip2mm

\begin{example}\label{Example 4.4}
Let $V$ be an $n$-dimensional vector space over a field $k$ and let $A$ be the exterior algebra constructed from $V$. Then by Example $1.5$ in \cite{Pollack}, $\text{Aut}_k(A)\cong \text{GL}_n\ltimes U$, where $U$ is a unipotent group. Therefore, $\text{Aut}_k(A)$ is $k$-rational because unipotent groups are rational, hence $R$-trivial.
\end{example}

\vskip2mm

\begin{example}\label{Example 4.5}
Let $A=\displaystyle\frac{k[X_1,X_2,\dots,X_n]}{\langle X_1,X_2,\dots,X_n\rangle^l}$ for some $l\geq 2$, then $\text{Aut}_k(A)\cong \mathrm{GL}_n\ltimes U$, where $U$ is a unipotent group (see \cite{AS1}, Theorem $1.3)$. Therefore, $\text{Aut}_k(A)$ is $k$-rational, hence $R$-trivial.
\end{example}

\vskip2mm

\begin{example}\label{Example 4.6}
In \cite{AS1}, it has been proved that if $A$ is a split local commutative associative monomial algebra, then the identity component $G_A$ has rank $n$ with a split torus of rank $n$. So, as a $k$-group, $G_A$ is $R$-trivial, as any $k$-split group is $k$-rational. 
\end{example}

\vskip2mm

\begin{proposition}\label{Proposition 4.7}
$($\cite{CM}, \emph{Proposition} $2.2)$ If any maximal torus of a connected reductive group $G$ over $k$ is rational over its field of definition (i.e., any maximal torus of $G_F$ is rational over $F$ for every field extension $F$ over $k$), then the group $G$ is itself rational. 
\end{proposition}

\vskip2mm
\begin{remark}\label{Remark 4.8}
For any connected algebraic group $G$ over a perfect field $k$, we have $G\cong_{b} G/U\times U$, where $U$ is the $k$-unipotent radical. Therefore, $G$ is rational if $G/U$ is rational, as $U$ is always rational. However, $G/U$ is a reductive group. Hence, if any maximal torus of $G_F$ is rational over $F$ for every field extension $F$ over $k$, then $G$ is also $k$-rational by Proposition \ref{Proposition 4.7} and Section $4$ of \cite{Voskrenskii}.
\end{remark}
\vskip2mm
\begin{definition}\label{Definition 4.9}
A grading on an algebra $A$ is a family of vector subspaces $(A_n)_{n\in\mathbb{Z}}$ of $A$ such that $A=\bigoplus_{n\in\mathbb{Z}} A_n$ and $A_n.A_m\subseteq A_{n+m}$ for all $n,m\in \mathbb{Z}$. The grading is called \emph{by the radical} when $A_m=0$ for $m<0$ and $J^m=\bigoplus_{n\geq m} A_n$, for each $m\geq 0$ (see \cite{Saorin}, Proposition $1.1$ for equivalent definitions$)$.  For any associative algebra $A$ with $J^0=A$, we can associate a graded by the radical algebra, $G_{r}(A)\coloneq \bigoplus_{m\geq 0} J^m/J^{m+1}$ (\emph{associated graded by the radical algebra}), as a vector space, and where the multiplication $(a+J^{m+1}). (b+J^{n+1})=(ab+J^{m+n+1})$ for $a\in J^{m}$ and $b\in J^{n}$, extends to all of $G_{r}(A)$ by linearity. An algebra $A$ is said to be \emph{graded by the radical algebra} if it has a grading by the radical, or equivalently, $A$ is isomorphic to $G_{r} (A)$. For more details, we refer to \cite{AS1} and \cite{Saorin}.
\end{definition}
\vskip2mm
\begin{remark}\label{Remark 4.10}
It is easy to see that there are well-defined maps for a split local algebra $A$ over $k$, which are given below. Since for any automorphism $\phi\in \mathrm{Aut}_{k}(A)$, it induces an automorphism $\bar{\phi}:J/J^2\rightarrow J/J^2$. We have a morphism of algebraic groups $\Phi_A: \mathrm{Aut}_{k}(A)\rightarrow \mathrm{GL}(J/J^2)$, where $\Phi_A(\phi)=\bar{\phi}$. If $\phi\in \mathrm{Aut}_{k}(A)$, then we can define a $k$-linear bijective map $\Psi_A(\phi)$, which is just the direct sum of the maps induced by $\phi$ on all $J^m/J^{m+1}$, $m\in \mathbb{N}\cup \{0\}$. It is clear that $\Psi_A(\phi)$ is actually an algebra automorphism, thus yields a map $\Psi_A:\mathrm{Aut}_{k}(A)\rightarrow \mathrm{Aut}_{k}(G_{r}(A))$. Similarly, for every $m\in\{2,\dots,n\}$, where $J^{n+1}=0$, we have a canonical map $\phi_{k}: \mathrm{Aut}_{k}(A)\rightarrow \mathrm{Aut}_{k}(A/J^m)$ that takes $\phi\in \mathrm{Aut}_{k}(A)$ to the induced automorphism $\tilde{\phi}: A/J^m\rightarrow A/J^m$.  It is easy to see that we have similar kind of maps $\Phi_{G_{r}(A)}: \mathrm{Aut}_{k}(G_{r}(A))\rightarrow \mathrm{GL}(J/J^2)$ and $\Phi_{A/J^{m}}:\mathrm{Aut}_{k}(A/J^m)\rightarrow \mathrm{GL}(J/J^2)$.  All of the above maps $(\Phi_A, \Psi_A, \phi_{k},\Phi_{G_{r}(A)}, \Phi_{A/J^m}$) are morphisms of algebraic groups.
\end{remark}
\vskip2mm
\noindent
The following result is known from \cite{Pollack}. We include a proof here for convenience.

\vskip2mm
\begin{proposition}\label{Proposition 4.11}
Let $A$ be a split local associative algebra over $k$ and  $G_r(A)$ is defined as above, i.e., $G_{r}(A)\coloneq (k\oplus\bigoplus_{m=1}^{n}(J^m/J^{m+1}))$ be the associated graded by the radical algebra of $A$. Consider the following commutative diagrams $(1)$ and $(2)$:
\begin{equation}\label{eq6}
\begin{minipage}{0.45\linewidth}
\begin{tikzcd}
    \mathrm{Aut}_k(A) \arrow[dr, "\Phi_A"] \arrow[d, "\Psi_A"']\\
    \mathrm{Aut}_k(G_{r}(A)) \arrow[r,"\Phi_{G_{r}(A)}"'] & \mathrm{GL}(J/J^2)
\end{tikzcd}
\begin{center}
\textbf{(1)}
\end{center}
\end{minipage}
\begin{minipage}{0.45\linewidth}
\begin{tikzcd}
   \mathrm{Aut}_k(A) \arrow[dr, "\Phi_A"] \arrow[d, "\phi_k"']\\
    \mathrm{Aut}_k(A/J^m) \arrow[r,"\Phi_{A/J^m}"'] & \mathrm{GL}(J/J^2)
\end{tikzcd}
\begin{center}
\textbf{(2)}
\end{center}
\end{minipage}
\end{equation}
Then each of the induced morphisms $(\Phi_A, \Phi_{G_{r}(A)}, \Phi_{A/J^m})$ has a unipotent kernel. Therefore, the rank of $\emph{\text{Aut}}_k(A)$ is bounded by the dimension of $J/J^2$.
\end{proposition}
\vskip1mm
\begin{proof}
We will prove that $\text{Ker}(\Phi_A)$ is unipotent; other cases follow similarly. Since $A$ is a split local algebra, $\text{Aut}_k(A)=\text{Aut}_k(J)$. Therefore, \[\text{Ker}(\Phi_A)=\{\phi\in \text{Aut}_k(J):\phi(x)-x\in J^2, \forall x\in J\}=\{\phi\in \text{Aut}_k(J):(\phi -Id)(J)\subset J^2\}.\] Let $x\in J$ and $(\phi-Id)(x)=\sum_{i=1}^n x_iy_i$, where each $x_i,y_i\in J$. Then,
\begin{align*}
(\phi-Id)\bigl((\phi -Id)(x)\bigl)={} & 
\sum_{i=1}^n\phi(x_iy_i)-\sum_{i=1}^nx_iy_i
\end{align*}
which implies,
\begin{flalign*}
 (\phi-Id)^2(x) ={}&
  \!\begin{aligned}[t] 
    &\sum_{i=1}^{n}\phi(x_i)\phi(y_i)-x_1\phi(y_1)+x_1\phi(y_1)-\dots\\ 
     &-x_n\phi(y_n)+x_n\phi(y_n)-\sum_{i=1}^nx_iy_i
  \end{aligned}\\
  ={}&
  \!\begin{aligned}[t]
    &\sum_{i=1}^n(\phi-Id)(x_i)\phi(y_i)+\sum_{i=1}^nx_i(\phi-Id)(y_i) \\
  \end{aligned}\\
\end{flalign*}
\noindent
However, $(\phi-Id)(x)\in J^2, \forall x\in J$. Hence, $(\phi-Id)^2(J)\subset J^3.$ Suppose $J^m=0$ for some $m\geq 2$. So, we can iterate the process and show that $(\phi-Id)^{m-1}=0.$ Since the kernel is a unipotent group, any maximal torus of $\text{Aut}_k(A)$ embeds in $\text{GL}(J/J^2)$. Hence, the rank of any maximal torus is less than or equal to $\text{dim}(J/J^2)$.
 \end{proof}
 \vskip2mm
 \begin{remark}\label{Remark 4.12}
Restricting $\Phi_A$ to $ G_A$ yields a homomorphism from $ G_A$ to $\text{GL} (J/J^2)$. By a similar line of argument, we can prove that $\text{Ker}(\Phi_A|_{G_{A}})$ is a unipotent group. Similarly, we can apply restrictions to other mappings to their identity components. Still, the proposition remains valid. To avoid notational ambiguity, the notation used for the identity component mappings will be the same as $\Phi_A$ and $\Psi_A$.
 \end{remark}

\vskip2mm
\begin{proposition}\label{Proposition 4.13}
Let $A$ be a split local associative algebra over $k$ such that $\emph{rank}(G_A)=\emph{dim}(J/J^2)=n$, then $G_A$ is $k$-rational.
\end{proposition}
\vskip1mm
\begin{proof}
We have the canonical map 
\begin{equation*}
\Phi_A: G_A \rightarrow \text{GL}(J/J^2)
\end{equation*}
whose kernel is unipotent by Proposition \ref{Proposition 4.11}. Therefore, any maximal torus of $G_A$ is isomorphic to a maximal torus of $\text{GL}(J/J^2)$. However, we know that any maximal torus of $\text{GL}_n$ over $k$ is isomorphic to \[R_{L_{1}/k}(\mathbb{G}_{m, L_1})\times R_{L_{2/k}}(\mathbb{G}_{m, L_{2}})\times\dots \times R_{L_{s}/k}(\mathbb{G}_{m,L_s}),\] where $L_i$'s are separable extensions of $k$ and $\sum_{i=1}^s[L_i:k]=n$. It follows from Section $3.12$ of \cite{Voskrenskii} that the Weil restriction of $\mathbb{G}_m$ is rational. Similarly, for any extension $F/k$, all maximal tori of $G_{A_F}$ are rational over $F$ if $\text{rank}(G_A)=\text{dim }(J/J^2)$. Hence, the result follows from Remark \ref{Remark 4.8}.
\end{proof}
\vskip2mm
\begin{proposition}\label{Proposition 4.14}
Let $A$ be a split local associative algebra satisfying $\emph{dim}(J/J^2)\leq 3$, then $G_A$ is $k$-rational.
\end{proposition}
\vskip1mm
\begin{proof}
Let $\Phi_A$ be the canonical map $\Phi_A:G_A\rightarrow \text{GL}(J/J^2)$ and consider 
\begin{equation*}\label{eq8}
\text{GL}(J/J^2)\xrightarrow{\text{det}} \mathbb{G}_m.
\end{equation*} 
If the $\text{dim}(J/J^2)$ is $1$ or $2$, then the result follows from the fact that any torus of rank at most $2$ is rational.
Let $\text{dim}(J/J^2)=3$ and $T$ be any maximal torus of $G_A$, then it embeds in $\text{GL}_3$ by Proposition \ref{Proposition 4.11}. Therefore, $\text{dim}(T)\leq 3$. If $\text{dim}(T)=1,2$, then $T$ is rational. Again if $\text{dim}(T)=3$, then $T$ is isomorphic to a maximal torus of $\mathrm{GL}_3$, which is rational by Proposition \ref{Proposition 4.13}. Therefore, $T$ is rational. Similarly for any field extension $F/k$, any maximal torus of $G_{A_F}$ is rational over $F$ if $\text{dim}(J/J^2)\leq 3$. Now by Remark \ref{Remark 4.8} we are done.
\end{proof}
\vskip2mm

\begin{corollary}\label{Corollary 4.15}
Let $A$ be a split local algebra with $\mathrm{dim}_k A\leq 5$. Then, $G_A$ is $k$-rational.
\end{corollary}
\vskip1mm
 \begin{proof}
 If $J^2=0$, then $\text{Aut}_k(A)=\text{Aut}_k(J)=\text{GL}(J)$. So, $G_A$ is rational. Suppose $J^2\neq0$. Since $A$ is a local algebra with $\text{dim }A\leq 5$, we have $\text{dim }J\leq 4$ and $1\leq \text{dim }J^2\leq 4$ (if $J^2=J$ then $J=0$). Hence, we get $\text{dim}(J/J^2)\leq 3$. Using Proposition \ref{Proposition 4.14}, we see that $G_A$ is $k$-rational.
\end{proof}
\vskip2mm
\noindent
We will often need the following result:
\vskip2mm
\begin{lemma}\label{Lemma 4.16}
$($\cite{CP}, \emph{Lemma} $1)$ Let $G$ be a connected algebraic group defined over an arbitrary field $k$. If $H$ is a connected subgroup of $G$ defined over $k$ such that for any field extension $F/k$, the kernel of the natural morphism $H^1(F,H)\rightarrow H^1(F,G)$ is trivial, then the variety 
\begin{equation*}\label{eq9}
G\cong_{b} G/H \times H.\end{equation*}
\end{lemma}
\vskip2mm
\begin{remark}\label{Remark 4.17}
Let $G$ be a connected algebraic group over a perfect field $k$ and $U$ be a connected $k$-unipotent subgroup of $G$. Then $G\cong_b G/U\times U$ by using Lemma \ref{Lemma 4.16}, as $H^{1}(F, U)$ is trivial for the unipotent group $U$ for every field extension $F$ over the perfect field $k$.  
\end{remark}
\vspace{0.1in}
\noindent
The above results motivate us to ask whether every connected algebraic group of the form $G_A$, where $A$ is a finite-dimensional associative algebra over a field $k$, is $R$-trivial or not. Below, we prove a key result on $R$-triviality of $G_A$.

\vskip2mm
\begin{theorem}\label{Theorem 4.18}
Let $A$ be a split finite-dimensional associative algebra over a perfect field $k$ with $A_s$ a semisimple subalgebra of $A$ and $J$ is the Jacobson radical of $A$ such that $A=A_s\oplus J$ (vector space decomposition), and $G_{A, A_s}\coloneq\{\sigma\in G_A: \sigma(a)=a \quad \forall a\in A_s\}$. Then, $G_A$ is $R$-trivial if and only if $G_{A,A_s}$ is $R$-trivial. Suppose any one of the following holds: 
\begin{enumerate}
    \item $J^2=0$;
    \vskip2mm
    \item $\emph{dim}(J/J^2)\leq 3$;
    \vskip2mm
    \item $\emph{rank}(G_{A,A_s})=\emph{dim}(J/J^2)$;
\end{enumerate}
\vskip2mm
 Then $G_A=\emph{\text{Aut}}_k(A)^{0}$ is $R$-trivial.
\end{theorem}
\vskip1mm
\begin{proof}
Using the decomposition from Corollary \ref{Corollary 2.2}, we can write $G_A$ as a product of some useful subgroups as Pollack did in \cite{Pollack}. Under the given hypothesis, we prove that these groups are all $R$-trivial. First, we prove some lemmas to decompose the group $G_A$.

\vskip2mm
\begin{lemma}\label{Lemma 4.19}$($\cite{Pollack}$)$
$G_A=\hat{J}G_{A}^{A_s}$, where $\hat{J}=\{\emph{Int}(1+r)\in G_A : r\in J\},\text{ and } G_{A}^{A_s}\coloneq\{\sigma \in G_A: \sigma(A_s)=A_s\}.$
\end{lemma}

\vskip1mm
\begin{proof}
Let $\sigma \in G_A$, then $A_s\oplus J=A=\sigma(A)=\sigma(A_s)\oplus J$. Therefore, using Corollary \ref{Corollary 2.2} we have $\sigma(A_s)=\rho(A_s)$, for some $\rho \in \hat{J}$. Thus, $\rho^{-1}\sigma\in G_{A}^{A_{s}}$ and it follows that $G_A=\hat{J}G_{A}^{A_s}$. Hence, \begin{equation*}\label{eq10}
\displaystyle\frac{G_A}{\hat{J}}\cong \displaystyle\frac{G_{A}^{A_s}}{G_{A}^{A_s}\cap \hat{J}}.
\end{equation*}
Note that $\text{Int}(1+J)=\hat{J}$ is connected unipotent (rational). It is clear from Lemma \ref{Lemma 4.16} that $G_A$ is $R$-trivial if and only if $G_A/\hat{J}$ is $R$-trivial. Hence, $G_A$ is $R$-trivial if and only if $\displaystyle\frac{G_{A}^{A_s}}{G_{A}^{A_s}\cap \hat{J}}$ is $R$-trivial. Again using Lemma \ref{Lemma 4.16}, we can conclude that $G_A$ is $R$-trivial if and only if $G_{A}^{A_s}$ is $R$-trivial, as $G _{A} ^ {A_s} \cap \hat{J}=G_{A,A_s}\cap \hat{J}$ is connected unipotent (image of a linear subspace of $J$).
\end{proof}

\vskip2mm

\begin{remark}\label{Remark 4.20}$($\cite{Pollack}$)$
If $J^2=0$, then for $G_A=\hat{J}G_{A}^{A_s}$, the product is a semidirect product.
\end{remark}

\vskip2mm
\begin{lemma}\label{Lemma 4.21}
$G_{A}^{A_s}=\emph{Inn}_{A}(A_{s}^{\times})G_{A,A_s}$, where $\emph{Inn}_{A}(A_s^{\times})\coloneq\{\emph{Int}(x)\in G_{A} : x\in A_{s}^{\times}\}$ and $G_{A,A_s}=\{\sigma\in G_A: \sigma|_{A_s}=Id\}$ and $A_s^{\times}=\emph{Units of }A_{s}$. Moreover, $G_{A}^{A_s}$ is $R$-trivial if and only if $G_{A,A_s}$ is $R$-trivial.
\end{lemma}
\vskip1mm
\begin{proof}
 Since $G_A^{A_s}$ is connected, the restriction morphism $\mathrm{res} : G_A^{A_s} \to \mathrm{Aut}_k(A_s)$ maps into $\mathrm{Aut}_k(A_s)^0$. Because $A_s \cong \bigoplus_{i=1}^m M_{n_i}(k)$, the Skolem–Noether theorem implies $\mathrm{Aut}_k(A_s)^0 = \mathrm{Inn}(A_s^\times)$. Thus, for every $\sigma \in G_A^{A_s}$, there exists $a \in A_s^\times$ such that $\sigma|_{A_s} = \mathrm{Int}(a)|_{A_s}$, which means $\tau = \operatorname{Int}(a)^{-1} \circ \sigma \in G_{A, A_s}$. Hence $G_A^{A_s} = \mathrm{Inn}_{A}(A_s^\times)G_{A, A_s}$.
\vskip2mm
\noindent
Since $G_{A, A_s} =\mathrm{Ker}(\mathrm{res})$ is normal in $G_A^{A_s}$, $A_s^\times$ acts on $G_{A, A_s}$ by conjugation via $\mathrm{Int}$: for $a \in A_s^\times$ and $\tau \in G_{A, A_s}$, the automorphism $a \cdot \tau := \mathrm{Int}(a) \circ \tau \circ \mathrm{Int}(a)^{-1}$ lies in $G_{A, A_s}$.  Consider the group $A_s^\times \ltimes G_{A, A_s}$ with multiplication given by $(a_1, \tau_1) * (a_2, \tau_2) := \bigl(a_1 a_2, \, (a_2^{-1} \cdot \tau_1) \circ \tau_2\bigr).$ 

\vskip2mm
\noindent
Therefore, we have the following surjective homomorphism of algebraic groups over $k$:\[\gamma: A_s^\times \ltimes G_{A, A_s} \longrightarrow G_A^{A_s}, \quad (a, \tau) \longrightarrow \mathrm{Int}(a) \circ \tau.\] Now kernel of the above map is given by $\mathrm{Ker}(\gamma)=\{(a,\tau):\text{Int}(a)\circ\tau=Id\}$. Now it is easy to see that $\mathrm{Ker}(\gamma)=\{(z,\mathrm{Int}(z^{-1})):z\in Z(A_s)^{\times}\}\cong Z(A_s)^{\times}\cong \mathbb{G}_{m,k}^m$. Therefore, we have a short exact sequence \[1 \longrightarrow \mathrm{Ker}(\gamma) \longrightarrow A_s^\times \ltimes G_{A, A_s} \xlongrightarrow{\mu} G_A^{A_s} \longrightarrow 1.\] As $H^1(F,\mathbb{G}_{m,k}^m)=\{1\}$ for every extension $F/k$ by Lemma \ref{Lemma 4.16} we have $$A_s^\times \times G_{A, A_s} \cong_b \frac{A_s^\times \ltimes G_{A, A_s}}{\mathrm{Ker}(\gamma)} \times \mathrm{Ker}(\gamma) \cong G_A^{A_s} \times \mathbb{G}_{m,k}^m.$$ Hence $G_{A}^{A_s}$ is $R$-trivial if and only if $G_{A,A_s}$ is $R$-trivial.
\end{proof}
\vskip2mm
\noindent
So, combining Lemma \ref{Lemma 4.19} and Lemma \ref{Lemma 4.21}, we get that $G_A$ is $R$-trivial if and only if $G_{A,A_s}$ is $R$-trivial.
\vskip2mm
\noindent
Since $J$ is an ideal of $A$, it may be considered as a $A_s$-bimodule or equivalently, as a left module for the semisimple algebra $A_{s}^{\#}\coloneqq A_s\otimes A_{s}^{\text{}op}$, where $A_{s}^{\text{}op}$ is the opposite algebra of $A_s$. Therefore, $J$ is an $A_{s}^{\#}$-algebra. We have for any $\sigma\in G_{A,A_s}$, $\sigma(arb)=a\sigma(r)b$, where $a,b\in A_s$ and $r\in J$. So, $G_{A,A_s}$ acts as an $A_s^{\#}$-algebra automorphism of $J$, and we have an injection $\Gamma: G_{A,A_s}\longrightarrow \text{Aut}_{A_s^{\#}-\textit{algebra}}(J)$. 

\vskip2mm
\begin{lemma}\label{Lemma 4.22}
$G_{A,A_s}\cong \emph{\text{Aut}}_{A_s^{\#}-\textit{algebra}}(J)^{0}$.
\end{lemma}
 \vskip1mm
\begin{proof}
It is clear from the previous argument that we have an injection 
\begin{equation*}\label{eq13}
\Gamma: G_{A,A_s}\longrightarrow \text{Aut}_{A_s^{\#}-\textit{algebra}}(J)^{0}
\end{equation*}
However, for $\psi\in \text{Aut}_{A_s^{\#}-\textit{algebra}}(J)^{0}$,  define $\bar \psi(a+r)=a+\psi(r)$, where $a\in A_s$ and $r\in J$, then $\Gamma(\bar{\psi})=\psi$ and the above map is surjective. Since $\Gamma(\psi_1\circ\psi_2)=\Gamma(\psi_1)\Gamma(\psi_2)$, we have the isomorphism.
\end{proof}

\vskip2mm
\noindent
So, it is enough to prove that $\text{Aut}_{A_{s}^{\#}-\text{algebra}}(J)^{0}$ is $R$-trivial. Hence, the $R$-triviality of $G_A$ depends on the radical of $A$. Since $A$ is an artinian algebra, there must exist an integer $n$ such that $J^n=0$. Next, we consider the case $n=2$.

\vskip2mm
\begin{lemma}\label{Lemma 4.23}
If $J^2=0$, then $G_{A,A_s}\cong \emph{\text{Aut}}_{A_s^{\#}-\textit{module}}(J)^{0}$ and it is a rational group.
\end{lemma}
\vskip1mm
\begin{proof}
For $J^2=0$, we have $\text{Aut}_{A_{s}^{\#}-\textit{algebra}}(J)^{0}\cong \text{Aut}_{A_s^{\#}-\textit{module}}(J)^{0}$, where $J$ is a $A_s^{\#}$-module. It follows from \cite{Pierce} or \cite{Pollack} that the $A_s^{\#}$-\text{module} $J$ can be written as 
$J=\sum_{i=1}^{\alpha} \lambda_{i}M_{i}$,
where this $M_i$'s are inequivalent irreducible $A_s^{\#}$-modules and $\lambda_i\in \mathbb{N}$. By using Schur's lemma it follows that $\text{Aut}_{A_s^{\#}-\text{module}}(\lambda_i M_i)=\text{GL}_{\lambda_i}$. However,
\begin{equation*}\label{eq15}
\text{Aut}_{A_s^{\#}-\textit{module}}(J)^{0}\cong \prod_{i=1}^{\alpha}\text{Aut}_{A_s^{\#}-\text{module}}(\lambda_{i}M_i), \text{ where $\prod$ denotes the product of groups. }
\end{equation*}
Therefore, we get \[G_{A,A_s}\cong \prod_{i=1}^{\alpha} \text{GL}_{\lambda_i}.\]

\vskip2mm
\noindent
Since $\text{GL}_n$ is rational, and a product of rational groups is rational, it follows that $G_{A, A_s}$ is rational. Therefore, $k$-rational points of $G_{A}=\text{Aut}_k(A)^{0}$ are $R$-equivalent over $k$. If we consider the $F$-rational points $\text{Aut}_k(A)^{0}(F)=\text{Aut}_F(A\otimes F)^{0}$, where $F$ is an extension of $k$, then the assertion follows easily by using Remark \ref{Remark 3.5}. Hence, $F$-rational points are also $R$-equivalent. Therefore, $\text{Aut}_k(A)^{0}=G_A$ is $R$-trivial if $J^2=0$.
\end{proof}

\vskip2mm

\begin{lemma}\label{Lemma 4.24}
If $J^2\neq 0$ and $\emph{dim}(J/J^2)\leq 3$ or $\emph{rank}(G_{A,A_s})=\emph{dim}(J/J^2)$, then $G_{A,A_s}$ is $R$-trivial.
\end{lemma}
\vskip1mm
\begin{proof} Since $J^2\neq 0$ and $G_{A,A_s}\cong \text{Aut}_{{A_s^{\#}}-\text{algebra}}(J)^{0}$, we have the canonical homomorphism \[G_{A,A_s}\rightarrow \text{Aut}_k(J/J^2).\] Now using similar argument as in Proposition \ref{Proposition 4.11}, we can show that any maximal torus of $G_{A,A_s}$ embeds in $\text{Aut}_k(J/J^2)=\text{GL}(J/J^2)$. Therefore, by using a similar line of argument as done in Proposition \ref{Proposition 4.14}, we can show that $G_{A, A_s}$ is $k$-rational if $\text{dim}(J/J^2)\leq 3$. If $\text{rank}(G_{A,A_s})=\text{dim}(J/J^2)$, then by Proposition \ref{Proposition 4.13}, $G_{A,A_s}$ is $k$-rational. Hence in any case $G_{A,A_s}$ is $R$-trivial.
\end{proof}
\vskip2mm
\noindent
Finally, we use Lemma \ref{Lemma 4.24} and Lemma \ref{Lemma 4.23}  to conclude that $G_A$ is $R$-trivial when it satisfies the conditions given in Theorem \ref{Theorem 4.18}.
\end{proof}

\vskip2mm
\begin{corollary}\label{Corollary 4.25}
Let $\emph{\text{Aut}}_k(A)$ be a connected reductive group, where $A$ is a split finite-dimensional associative algebra, then $\mathrm{Aut}_k(A)$ is $R$-trivial.
\end{corollary}
\vskip1mm
\begin{proof}
For any connected reductive group of the form $\text{Aut}_k(A)$, we have $J^2=0$ by Corollary \ref{Corollary 3.2}, where $J$ is the radical of $A$. It follows from Theorem \ref{Theorem 4.18} that $\text{Aut}_k(A)$ is $R$-trivial.
\end{proof}

\vskip2mm
\begin{corollary}\label{Corollary 4.26}
Let $T$ be a non $R$-trivial torus over $k$. Then, there does not exist any split finite-dimensional associative algebra $A$ over $k$ such that $\emph{\text{Aut}}_k(A)\cong T$.
\end{corollary}
\vskip 1mm
\begin{proof}
Follows from Corollary \ref{Corollary 4.25}.
\end{proof}
\vskip2mm
\begin{corollary}\label{Corollary 4.27}
Let $A$ be a split finite-dimensional associative algebra over $k$. For every field extension $F$ over $k$, let us define $R{G_A}(F)$, \[R{G_A}(F)\coloneq\{\sigma\in G_A(F) : \sigma \text{ is R-equivalent to identity } \} .\] Then, $R{G_A}(F)=\{Id\}$ for all field extension $F$ over $k$ if and only if $A\cong \bigoplus_i k$.
\end{corollary}
\vskip1mm
\begin{proof}
It is already known that $R{G_A}(F)$ is a normal subgroup of $G_A(F)$. Moreover, it is clear that any unipotent element is $R$-equivalent to the identity. Therefore, for every field extension $F$ over $k$, we have $U(F)\subset R{G_A}(F)$, where $U$ is the $k$-unipotent radical of $G_A$. Suppose $R_{G_A}(F)=\{Id\}$ for all field extension $F$ over $k$, then $k$-unipotent radical $U$ of $G_A$ is trivial. Hence, $G_A$ is reductive over the perfect field $k$. However, by Corollary \ref{Corollary 4.25}, $G_A$ is $R$-trivial, which means $G_A(F)=R{G_A}(F)$ for every field extension $F$ over $k$. Hence, we get that $G_A=\{Id\}$ which further implies that $A\cong \bigoplus_i k$. This completes one direction of the proof, and the other direction follows easily.
\end{proof}
\vskip2mm
\begin{remark}\label{Remark 4.28}
It follows directly from Theorem \ref{Theorem 4.18} that the groups $G_A$ considered in Example \ref{Example 4.2} ($J=0$), Example \ref{Example 4.4}, Example \ref{Example 4.5}, and Example \ref{Example 4.6} ($\text{rank}(G_A)=\text{rank}(G_{A,A_s})=\text{dim}(J/J^2)$) are all $R$-trivial.
\end{remark}
\vskip2mm
\noindent
To state our next result, we need the following lemma.

\vskip2mm

\begin{lemma}\label{Lemma 4.29}
$($\cite{Pollack}, \emph{Theorem} $1.6)$ Let $A$ be a split local associative algebra over a perfect field $k$ such that $G_A$ is nilpotent. Then, either $G_A$ is a torus or unipotent. It can not be a direct product of a torus with a non-trivial unipotent group.   
\end{lemma}

\vskip2mm

\begin{theorem}\label{Theorem 4.30}
Let $A$ be a split finite-dimensional associative algebra over a perfect field $k$. Assume that $\emph{\text{Aut}}_k(A)$ is a nilpotent group. Then, $G_A=\emph{\text{Aut}}_k(A)^{0}$ is $k$-rational.
\end{theorem}
\vskip1mm
\begin{proof}
Since $\text{Aut}_k(A)$ is nilpotent, it is a solvable group, which further implies that $\text{Inn}_{A}(A_s^{\times})$ is solvable, where $A_s$ is a semisimple subalgebra of $A$ with $A=A_s\oplus J$. Hence, we get that $A_s^{\times}$ is solvable. However, $A_s\cong \bigoplus_{i=1}^{m}M_{n_i}(k)$ and $A_s^{\times}$ solvable implies $A_s\cong \bigoplus_{i}^{m} k$. So, we can decompose $A\cong \bigoplus_{i}^{m} A_i$ as $G_A$ is nilpotent, where each $A_i$ is a split local algebra, and this decomposition is unique (see \cite{Pollack}, \cite{Pierce}).
It follows that $G_A\cong\prod_{i}^{m}G_{A_{i}}$. Hence, to prove the assertion, it is enough to prove that $G_A$ is $k$-rational when $A$ is a split local algebra and $G_{A}$ is a nilpotent group. A connected nilpotent algebraic group is a direct product of a torus with a unipotent group. However, using the previous Lemma \ref{Lemma 4.29}, we have either $G_A=T$, where $T$ is the torus, or $G_A$ is unipotent. If $G_A$ is a torus, then by Theorem \ref{Theorem 3.4}, $G_A$ is rational. On the other hand, if $G_A$ is a unipotent group, then it is rational as a variety. Therefore, the result follows.
\end{proof}

\vskip2mm
\begin{proposition}\label{Proposition 4.31}
For a split finite-dimensional associative algebra $A$ over a perfect field $k$, the $R$-triviality of $G_{A}$ is Morita invariant.
\end{proposition}
\vskip1mm
\begin{proof}
Corollary $19$ in \cite{AS3} shows that $G_{A,A_s}/(G_{A,A_s}\cap \hat{J})$ is a Morita invariant, where $\hat{J}=\{\text{Int}(1+r):r\in J\}$. In Theorem \ref{Theorem 4.18}, we have shown that $G_A$ is $R$-trivial if and only if $G_{A, A_{s}}$ is $R$-trivial. Since $G_{A,A_s}\cap \hat{J}$ is a connected unipotent group, it follows from Lemma \ref{Lemma 4.16} that, \[G_{A,A_s}\cong_{b} \{G_{A,A_s}/(G_{A,A_s}\cap \hat{J})\} \times (G_{A,A_s}\cap \hat{J}).\] Therefore, it is clear that $G_A$ is $R$-trivial if and only if $G_{A,A_s}/(G_{A,A_s}\cap \hat{J})$ is $R$-trivial. Hence, the $R$-triviality of $G_{A}$ is Morita invariant. That means if $A$ and $B$ are two Morita equivalent algebras over $k$, then $G_{A}$ is $R$-trivial if and only if $G_{B}$ is $R$-trivial. 
 \end{proof}

\vskip2mm
\begin{remark}\label{Remark 4.32}
Theorem \ref{Theorem 4.18} shows that the $R$-triviality of $G_A$ depends on $G_{A, A_s}$ for a split finite-dimensional associative algebra. However, from Proposition \ref{Proposition 4.31}, it is clear that if $A$ and $B$ are two unital split finite-dimensional associative algebras which are Morita equivalent, then the $R$-triviality of $G_{A}$ implies the $R$-triviality of $G_{B}$ and vice versa. It follows from the Wedderburn-Malcev decomposition that any split finite-dimensional unital associative algebra $A$ over a perfect field $k$ is isomorphic to $A_s \oplus J$, where $A_s \cong \bigoplus_{i=1}^{m}M_{n_{i}}(k)$ and $J$ is the radical of $A$. Again, from Theorem \ref{Theorem 2.5} and Section $6.6$ of \cite{Pierce}, we know that for any such algebra $A$, there exists a Morita equivalent algebra $B=B_s\oplus J^{\prime}$, where $B_s\cong \bigoplus_{i=1}^{m} k$ and $J^{\prime}$ is the radical of $B$. Therefore, studying the $R$-triviality of $G_A$ reduces to studying split basic algebras, and we can use the quiver and relation representation of split basic algebras to investigate the $R$-equivalence property of $G_A$.
\end{remark}
 
\vskip2mm
\begin{proposition}\label{Proposition 4.33}
Let $A$ be a split local associative algebra over a field with $\emph{\text{char }}k=0$, then $G_A$ is stably birational to $\emph{\text{Im}}(\Phi_A)$, where $\Phi_A$ is the canonical map, defined in Proposition \ref{Proposition 4.11}.
\end{proposition}
\vskip1mm
\begin{proof}
We have the map $\Phi_A: G_A\rightarrow \text{GL}(J/J^2)$, where the kernel is a connected unipotent group by Proposition \ref{Proposition 4.11}. Now, using the Lemma \ref{Lemma 4.16}, we can conclude that $G_A\cong_b G_A/U\times U$, where $U$ is $\text{Ker}(\Phi_A)$. Hence, $G_A\cong_b \text{Im}(\Phi_A)\times U$, and $G_A(F)/R\cong \mathrm{Im}(\Phi_A)(F)/R$ for every field extension $F/k$.
\end{proof}

\vskip2mm 
\noindent
Therefore, it is enough to study the image of the canonical map $\Phi_A$ for studying the rationality of $G_A$ for a split local associative algebra $A$. 

\section{Finite Dimensional Commutative Algebras}\label{sec: commutative algebra}
\vskip2mm
Let $A$ be a split finite-dimensional commutative algebra over $k$, then we have $A \cong\bigoplus_{i=1}^n A_i$, where $A_i$’s are split local algebras over $k$. Hence, $G_A\cong \prod_{i=1}^{n} G_{A_i}$ (see \cite{ISD}, \cite{AS1}). Therefore, to study rationality properties of $G_A$, where $A$ is a split commutative algebra, it is enough to focus on split local commutative algebras. From this point onward, our primary emphasis will be on the automorphism group of split finite-dimensional local commutative associative algebras. Using the representation of finite-dimensional associative algebras by quiver and relation, one knows that any such algebra $A$ is isomorphic to $k[X_1,\dots, X_n]/{I}$, where $I$ is an ideal of $k[X_1,\dots, X_n]$ such that there exists an $l\geq2$ with $\langle X_1,\dots, X_n \rangle^l \subseteq I \subseteq \langle X_1,\dots, X_n \rangle^2$ (see \cite{PG}, \cite{AS1}). The minimum $l$ is called the Loewy length of $A$, denoted by $L(A)$. Every such ideal is called an \emph{admissible} or \emph{adequate ideal} for $A$ in $k[X_1,\dots,X_n]$. A finite set of generators for an adequate ideal is called \emph{adequate sets of relations} for $A$. When $A$ is realized as $k[X_1,\dots,X_n]/I$, then $J^m=(\text{Jacobson}(A))^m$ gets identified with $[(X_1,\dots,X_n)^m+I]/I$ and $n = \text{dim}(J/J^2)$. Let $A$ be identified with $k[X_1,\dots,X_n]/I$, where $I$ is an admissible ideal. Suppose we denote $x_i=X_i+I$, then $\{\bar{x_{i}}=x_i+J^2:i\in\{1,\dots,n\}\}$ is a basis of $J/J^2$. By taking matrices with respect to that basis, $\text{GL}(J/J^2)$ gets identified with $\text{GL}_{n,k}$. The notion of change of variables is defined as follows.

\vskip1mm

\begin{definition}\label{Definition 5.1}$($\cite{AS1}$)$
A homomorphism of algebras $F:k[X_1,\dots,X_n]\longrightarrow k[X_1,\dots,X_n]$ is said to be a change of variables in $k[X_1,\dots,X_n]$, if there is a $n\times n$ matrix $(\lambda_{ij})\in \text{GL}_{n}(k)$ such that $F(X_j)=\sum_{1\leq i\leq n}\lambda_{ij}X_i$ modulo $(\langle X_1,\dots,X_n\rangle^2)$ for every $j\in \{1,\dots,n\}$. In particular, if no monomial of degree $\geq 2$ appears in the change of variables $F$, i.e., $F(X_j)=\sum_{1\leq i\leq n}\lambda_{ij}X_i$ for all $j\in\{1,\dots,n\}$, then we call such a change of variable a \emph{linear change of variables}, and $(\lambda_{ij})_{n\times n}$ is called the \emph{linear change of variables matrix} and denoted by $M_{F}=(\lambda_{ij})_{n\times n}$. For any linear change of variables $F$, we have $F(f)(X)=f(XM_{F})$, where $X=(X_1,\dots,X_n)$.
\end{definition}

\vskip2mm
\noindent
By using Theorem $1.5$ from \cite{Saorin} for a split commutative local algebra $A$, we get that $A$ is graded by the radical algebra if and only if $A$ admits a homogeneous admissible ideal $I$.

\vskip2mm

\begin{proposition}\label{Proposition 5.2}$($\cite{AS1}$)$  
Let A be a split local commutative algebra with Loewy length $l$ and $\emph{dim}(J/J^2) = n$. Suppose that $I, I'$ are two ideals of $k[X_1,\dots,X_n]$, the first being adequate for $A$. The following assertions are equivalent for a homomorphism of algebras
\begin{equation*}\label{eq19}
\phi: k[X_1,\dots,X_n]/I \longrightarrow k[X_1,\dots,X_n]/I':
\end{equation*}
\begin{enumerate}
\item $\phi$ is an isomorphism;
\vskip2mm
\item there is a change of variables $F$ in $k[X_1,\dots,X_n]$ such that $I'= F(I) + \langle X_1,\dots,X_n\rangle^l$ and $\phi(f+ I) = F(f) + I'$ for every $f\in k[X_1,\dots,X_n]$.  
\end{enumerate}
\vskip2mm
In particular, when $I = I'$, we get that every automorphism of $A$ is induced by a change of variables $F$ such that $F(I)\subset I$.
\end{proposition}
\vskip2mm
\noindent
We will denote by $C_{(A, I)}$ (resp. $CL_{(A, I)}$), the set of changes of variables (resp. linear changes of variables) $F$ of $k[X_1,\dots,X_n]$ satisfying $F(I)\subset I$. One knows that $C_{(A,I)}$ is a monoid and $CL_{(A,I)}$ is isomorphic to a subgroup $L_A$ of $\mathrm{Aut}_k(A)$ such that $L_A\cap \text{Ker}(\Phi_A)=\{1\}$ 
 (see \cite{AS3}, Corollary 21).
\vskip2mm
\noindent
Next, we review Lemma $22$ from \cite{AS3} in the context of a split local commutative algebra $A$ over $k$.

\vskip2mm

\begin{lemma}\label{Lemma 5.3}
$($\cite{AS3}, \emph{Lemma} $22)$ Let $A$ be a split local commutative algebra over $k$. Suppose $A\cong k[X_1,\dots,X_n]/I$ be a quiver representation with $I$ an \emph{admissible ideal} (a finite set of adequate relations). If $I$ is a homogeneous ideal (i.e., $A$ is a graded by the radical algebra), i.e, if homogeneous elements generate the \emph{adequate ideal} $I$, then
\begin{equation*}\label{eq20}
\emph{\text{ Im}}(\Phi_A )\cong CL_{(A, I)}\cong L_A.
\end{equation*}
\end{lemma}

\vskip2mm
\begin{proposition}\label{Proposition 5.4}
Let $A$ be a split local commutative graded by the radical algebra with a homogeneous admissible ideal $I$ and $\emph{dim}(J/J^2)=n$. Then, \[\emph{\text{Aut}}_k(A)\cong \emph{\text{Im}}(\Phi_A)\ltimes \emph{\text{Ker}}(\Phi_A), \text{ where $\Phi_A$ is the canonical map as in Proposition \ref{Proposition 4.11}. }\]
\end{proposition}
\vskip1mm
\begin{proof}
We have the canonical homomorphism $\Phi_A: \text{Aut}_k(A)\rightarrow \text{GL}(J/J^2)$ whose kernel is unipotent. Therefore, we have the following short exact sequence: 
\begin{center}
\begin{tikzcd}
1\arrow[r] & \text{Ker}(\Phi_A)\arrow[r] & \text{Aut}_k(A)\arrow[r] & \text{Im}(\Phi_A)\arrow[bend left=33]{l}{s}\arrow[r] & 1
\end{tikzcd}
\end{center}
It is known from Lemma \ref{Lemma 5.3} that any element of $\text{Im}(\Phi_A)$ gives an automorphism of $A$, which is induced by a linear change of variables. Hence, we have a section $s: \text{Im}(\Phi_A)\rightarrow \text{Aut}_k(A)$ given by $s(M)=\phi_M$, where $\phi_M(f+I)=f(XM)+I$, $X=(X_1,\dots,X_n)$. Finally, we have $\Phi_A\circ s=Id$, which further implies that the above sequence is split and provides the required isomorphism.
\end{proof}

\vskip2mm
\begin{proposition}\label{Proposition 5.5}
Let $A$ be a split local commutative graded by the radical algebra over $k$ with $\emph{dim}(J/J^2)=n\geq 1$, then $G_A$ is $k$-isotropic, i.e., there exists a $k$-embedding of $\mathbb{G}_m$ in $G_A$. This $\mathbb{G}_m$ is a central torus in $\emph{\text{Im}}(\Phi_A)$ and contained in all maximal tori of $\mathrm{Im}(\Phi_A)$. 
\end{proposition}
\vskip1mm
\begin{proof}
Let $A$ be graded by the radical algebra, which implies that there exists a homogeneous adequate ideal, $I$ of $A$. We know that any automorphism $\phi$ of $A$ is induced by a change of variables $F$ such that $\phi(f+I)=F(f)+I$ and $F(I)\subseteq I$. Consider the embedding $\Gamma :\mathbb{G}_m\rightarrow G_A$ given by $\Gamma(\alpha)=\phi_{\alpha}$, where $\phi_{\alpha}(\bar X_i)=\alpha \bar X_i$, $\bar{X_i}\coloneq X_i+I$. Therefore, for each $\phi_{\alpha}$, the corresponding change of variables is $F_{\alpha}(X_i)=\alpha X_i$, which stabilizes $I$. This is an injective homomorphism. If $\alpha\in \mathbb{G}_m$ is in $\text{Ker}(\Gamma)$, then $(\alpha-1)X_i\in I\subseteq \langle X_1,\dots,X_n\rangle^2$ $\forall i$, which is only possible if $\alpha=1$. Hence, the group is $k$-isotropic. Moreover $\mathbb{G}_m$ embeds in $\text{Im}(\Phi_A)$ as $\text{Ker}(\Phi_A)$ is unipotent.
\vskip2mm
\noindent
Let $\phi_1\in G_A$ be such that it is given by a linear change of variables, i.e., $\phi_1(f+I)=F_1(f)+I$, where $F_1$ is a linear change of variables corresponding to $\phi_1$, and let $\phi_2\in \mathbb{G}_m$ with $F_2(X_i)=\alpha X_i$ as the corresponding change of variables. Then, it is easy to verify that $F_1\circ F_2=F_2\circ F_1$, as $F_1(f)(X)=f(\sum_{i=1}^n\lambda_{i1}X_i,\dots,\sum_{i=1}^{n}\lambda_{in}X_i)$, where $(\lambda_{ij})_{n\times n}$ is a non-singular matrix. It follows from Lemma \ref{Lemma 5.3} that $\text{Im}(\Phi_A)$ is given by a linear change of variables. Hence, via the above embedding, $\mathbb{G}_m$ is a central torus in $\text{Im}(\Phi_A$). Moreover, for any maximal torus $T$ in $\mathrm{Im}(\Phi_A)$, $T.\mathbb{G}_m$ is again a torus in $\mathrm{Im}(\Phi_A)$ containing $T$. Therefore, $\mathbb{G}_m\subset T$.  
\end{proof}

\vspace{0.1in}
\noindent
In the rest of the article, we will investigate the following question for $G_A$, where $A$ is a split local finite-dimensional commutative algebra.

\vskip2mm
\begin{question}\label{Question 5.6}
Is the group $G_A$ always $R$-trivial, when $A$ is a split local commutative graded by the radical algebra?
\end{question}

\vskip2mm

\begin{lemma}\label{Lemma 5.7}
Let $A$ be a split local commutative algebra over $k$ with $\emph{dim}(J/J^2)=n$. By quiver and relations representation of $A$ we have an isomorphism $A\cong k[X_1,\dots,X_n]/I$, where $\langle X_1,\dots,X_n \rangle^l \subseteq I\subseteq \langle X_1,\dots,X_n \rangle^2$, $l\geq 2$, $I$ is the admissible ideal of $A$. Then, $I=\langle X_1,\dots,X_n \rangle^l+\langle P_1,\dots,P_m \rangle$, where $l$ is the Loewy length of $A$ and all non-zero $P_i$'s are polynomials of degree at least $2$ and at most $l-1$ with no linear component.
\end{lemma}
\vskip1mm
\begin{proof}
Let $f\in I$. Then $f=\sum_{i_1,\dots,i_n}a_{i_{1},\dots,i_{n}}X_{1}^{i_{1}}\dots X_{n}^{i_{n}}$. By the given hypothesis, we have $\langle X_1,\dots, X_n \rangle^l\subseteq I$; therefore, monomials with degree greater than or equal to $l$ are contained in $I$. Hence, we can write $f=\sum_{i}\lambda_{i}h_i+P$, where $h_i$'s are distinct monomials in \{$X_1,\dots,X_n$\} of degree at least $l$ and $P$ is a polynomial of degree at most $l-1$. However, $f-\sum_{i}\lambda_{i}h_{i}\in I$, therefore $P\in I$. Since $I$ is contained in $\langle X_1,\dots,X_n\rangle ^2$, $P$ must have degree at least $2$ with no linear component. Now, for each $f\in I$, we have a unique $P$ in $I$ which has the degree at least $2$ and at most $l-1$ with no linear component. Consider the set
\[S=\{P: P \text{ corresponds to each } f\in I \text{ with the above property }\}.\] Suppose $I$ is a proper ideal lying in between $\langle X_1,\dots,X_n\rangle^l$ and $\langle X_1,\dots,X_n\rangle ^2$. Then $S$ must be non-empty. Therefore, the ideal generated by $S$ is contained in $I$, and let us call it $I_1$. Moreover, $k[X_1,\dots,X_n]$ is a noetherian ring, so $I_1= \langle P_1,\dots,P_m \rangle$ where $P_i\in S$ and $I_1\subset I$. Hence $I=\langle X_1,\dots,X_n\rangle^l+\langle P_1,\dots,P_m\rangle$, where ${P_i}$'s have the property given in the hypothesis.
\end{proof}
\vskip2mm
\begin{definition}\label{Definition 5.8} 
Stabilizer subgroup of a polynomial $f\in k[X_1,\dots,X_n]$ in $\text{GL}(J/J^2)$, where $\text{dim}(J/J^2)=n$, is defined by 
\begin{equation*}\label{eq28}
\text{Stab}(f)\coloneq\{M\in \text{GL}(J/J^2): f(XM)=f(X), X=(X_1,\dots,X_n)\}.
\end{equation*}
\end{definition}
\noindent
$\text{Stab}^{0}(f)$ denotes the identity component of the stabilizer subgroup of $\text{GL}(J/J^2)$.

\vskip4mm
\begin{definition}\label{Definition 5.9}
The similarity subgroup of a polynomial $f\in k[X_1,\dots,X_n]$ in $\text{GL}(J/J^2)$, where $\text{dim}(J/J^2)=n$, is defined by \[\text{Sim}(f)\coloneq\{M\in \text{GL}(J/J^2): f(XM)=\alpha_M f(X), X=(X_1,\dots,X_n)\text{ and } \alpha_M\in k^{\times}\}.\]
\end{definition}
\noindent
$\text{Sim}^{0}(f)$ denotes the identity component of the similarity subgroup in $\text{GL}(J/J^2)$. It is easy to see that $\text{Sim}^{0}(f)=\mathbb{G}_m.\text{Stab}^{0}(f)$.
\vskip4mm
\begin{theorem}\label{Theorem 5.10}
Let $A$ be a split local commutative graded by the radical algebra over a perfect field $k$ with $\emph{\text{dim}}(J/J^2)\leq 4$, then $G_A$ is $k$-rational.
\end{theorem}
\vskip2mm
\begin{proof}
By Proposition \ref{Proposition 5.4}, it is enough to show that $\mathrm{Im}(\Phi_A)^0$ is $k$-rational. Suppose we have $\text{dim}(J/J^2)\leq 3$, then this case directly follows from Proposition \ref{Proposition 4.14}. If $\text{dim}(J/J^2)=4$ and let $T$ be a maximal torus of $\mathrm{Im}(\Phi_A)^0$, then $\text{dim }T\leq 4$ as $\mathrm{Im}(\Phi_A)^{0}\subset \mathrm{GL}_4$. If $\text{dim }T=4$, then $T$ is $k$-rational by Proposition \ref{Proposition 4.13} and if $\text{dim }T=1,2$ then also $T$ is $k$-rational. Now let $\text{dim }T=3$, then by Proposition \ref{Proposition 5.5}, we have $\mathbb{G}_m\subset T$, which implies $T/\mathbb{G}_m$ has rank $2$. However, $T\cong_b T/\mathbb{G}_m\times \mathbb{G}_m$ as $H^{1}(F,\mathbb{G}_m)=\{1\}$ for every field extension $F$ over $k$. Hence, in any case $T$ is $k$-rational. Similarly, for any field extension $F$ over $k$, we can show that any maximal torus of $\mathrm{Im}(\Phi_A)^{0}_F$ is rational over $F$ if $\text{dim}(J/J^2)\leq 4$. Hence $\mathrm{Im}(\Phi_A)^0$ is $k$-rational by Remark \ref{Remark 4.8}.

\end{proof}
\vskip2mm
\begin{proposition}\label{Proposition 5.11}
Let $A$ be a split local commutative graded by the radical algebra with $\emph{dim}(J/J^2)=n$. Suppose the corresponding admissible ideal for $A$ is $I=\langle X_1,\dots, X_n \rangle^l+\langle f\rangle$, where $2\leq \emph{\text{deg}}(f)< l$ is the Loewy length of $A$ and $f$ is a homogeneous polynomial.  Then,
\begin{equation*}\label{eq23}
\emph{Aut}_k(A)\cong \emph{Sim}(f)\ltimes \emph{\text{Ker}}(\Phi_A), where \text{ $\Phi_A$ is the canonical map as in Proposition \ref{Proposition 4.11}. }
\end{equation*}
\end{proposition}

\vskip1mm
\begin{proof}
We have the canonical map $\Phi_A: \text{Aut}_k(A)\rightarrow\text{GL}(J/J^2)$. It is enough to prove that for the given $I$, $\text{Im}(\Phi_A)=\text{Sim}(f)$. Let $M\in \text{Im}(\Phi_A)$, then $M$ induces an automorphism of  $A$ which is given by the linear change of variables $F_1(X)=XM$, where $X=(X_1,\dots,X_n)$ and $F_1(I)\subset I$ by Lemma \ref{Lemma 5.3}. As $f\in I$, \[F_1(f)(X)=f(\sum_{i}m_{i1}X_i,\dots,\sum_{i}m_{in}X_i)\in I,\] where $M=(m_{ij})_{n\times n}$ and $X=(X_1,\dots,X_n)$. Since a linear change of variables does not change the degree of the polynomial, we have $\text{deg}(F_1(f))=\text{deg}(f)$ and $F_1(f)$ is a homogeneous polynomial. As $f$ is a non-zero homogeneous polynomial, we have $F_1(f)=gf$, for some non-zero homogeneous polynomial $g\in k[X_1,\dots,X_n]$. However, $\text{deg}(f)=\text{deg}(F_1(f))$ implies $\text{deg}(g)=0$. Therefore,
$F_1(f)=\alpha f$ and $f(\sum_{i}m_{i1}X_i,\dots,\sum_{i}m_{in}X_i)=\alpha f(X_1,\dots,X_n)$ with $\alpha\in k^{\times}$.
Hence, $M\in \text{Sim}(f)$. So, we have $\text{Im}(\Phi_A)\subset \text{Sim}(f)$. It is easy to see that $\text{Sim}(f)\subset \text{Im}(\Phi_A)$. Now using Proposition \ref{Proposition 5.4}, $\text{Aut}_k(A)\cong \text{Sim}(f)\ltimes \text{Ker}(\Phi_A)$.
\end{proof}

\vskip2mm

\begin{theorem}\label{Theorem 5.12}
Let $A$ be a local commutative algebra over a perfect field $k$ with $\emph{dim}(J/J^2)=n$. Suppose $A\cong k[X_1,\dots, X_n]/I$ by quiver and relation representation of $A$, where $I=\langle X_1,\dots, X_n\rangle^l+\langle f\rangle$, $f$ is a homogeneous polynomial and $2\leq \mathrm{ deg}(f)<l$ is the Loewy length of $A$. Then, $G_A$ is $R$-trivial as a $k$-group if and only if $\mathrm{Sim}^{0}(f)$ is $R$-trivial.
\end{theorem}
\vskip1mm
\begin{proof}
Applying Proposition \ref{Proposition 5.11}, we get that $G_A\cong \text{Sim}^{0}(f)\ltimes \text{Ker}(\Phi_A)^{0}$, where $\text{Ker}(\Phi_A)^{0}$ is a unipotent group. Therefore, $G_A$ is $R$-trivial if and only if $\text{Sim}^{0}(f)$ is $R$-trivial.
\end{proof}

\vskip2mm
\begin{proposition}\label{Proposition 5.13}
$($\cite{CM}, \emph{Proposition} $2.4)$ Unitary, special orthogonal groups, and symplectic groups are rational.   
\end{proposition}

\vskip2mm
\begin{example}\label{Example 5.14}
Let $I$ be an admissible ideal of $k[X_1,\dots,X_{2n+1}]$, where $I=\langle X_1,\dots, X_{2n+1} \rangle^l+\langle q\rangle$, $q$ is a non-degenerate quadratic form of dimension $2n+1$, $\text{char}(k)\neq2$ and $l>2$. Then $G_A$ is rational for $A\cong k[X_1,\dots,X_{2n+1}]/I$ by using Theorem \ref{Theorem 5.12} and Proposition \ref{Proposition 5.13} as $\text{Sim}^{0}(q)\cong \mathbb{G}_m\times \text{SO}(2n+1,q)$ (\cite{KMRT}, Proposition $12.4$). 
\end{example}

\vskip2mm
\begin{definition}\label{Definition 5.15}
  $($\cite{OS}, \cite{Schneider}$)$  A non-zero homogeneous polynomial $f\in k[X_1,\dots, X_n]$ of degree $d$ is called non-singular where $\text{char }k=0 \text{ or } >d$ if any one of the following equivalent conditions holds,
    \begin{enumerate}
        \item $0=(0,\dots,0)$ is the unique common root of its partial derivatives $\displaystyle\frac{\partial f}{\partial
        X_1},\dots,\displaystyle\frac{\partial f}{\partial X_n}$ in $\bar{k}^n$. 
        \vskip2mm
        \item Discriminant $(\Delta_{n,d}(f))$ of $f$ does not vanish (see Chapter $13$, \cite{GKZ}).
        \vskip2mm
        \item Let $\Theta_f$ be the corresponding symmetric multilinear form of $f$ of degree $d$ over a vector space $V=\bar{k}^n$ of dimension $n$ over $\bar{k}$. Then $\Theta_f(v,\dots,v,w)=0$ for all $w\in V$ implies $v=0$. 
    \end{enumerate}
\end{definition}

\vspace{0.1in}
\noindent
Next, we recall the following theorem, originally due to Camille Jordan.
\begin{theorem}\label{Theorem 5.16}
$($\cite{Schneider}$)$ Let $\Theta$ denote a symmetric multilinear map of degree $r\geq 3$ on a vector space of dimension $n$ over a field $k$. Assume that $k$ has characteristic $0$ or greater than $r$. If $\Theta$ is nonsingular, then its orthogonal group (stabilizer group) in $\emph{GL}_n$ is finite.
\end{theorem}
\vskip2mm
\noindent
We now prove the following result.
\vskip 2mm
\begin{proposition}\label{Proposition 5.17}
Let $A$ be a split local commutative graded by the radical algebra with Loewy length $l$ and $\emph{dim}(J/J^2)=n$, where the admissible ideal $I$ is given by $\langle X_1,\dots,X_n \rangle^l +\langle f_1,\dots,f_m\rangle$, $f_i$'s are homogeneous polynomials such that not all $f_{i}$'s are zero. Let $d_i=\emph{deg}(f_i)$, where $1\leq i\leq m$ and $2\leq d_i<l$.  Suppose there exists a unique generator $f_s$ with degree $\emph{min}\{d_i: 1\leq i\leq m\}$. Then 
\[\bigcap_{i=1}^{m}\emph{Sim}(f_i)\subset\emph{\text{Im}}(\Phi_A)\subset \emph{Sim}(f_s).\]
In particular, if $f_s$ is a non-singular homogeneous polynomial of degree $\geq 3$ over a field $k$ with $\emph{char }k=0$ or $\emph{char }k=p>3$ and $f_s$ is non-singular polynomial of degree $d$ with $2<d<p$, then $G_A$ is a $k$-rational group of rank $1$.
\end{proposition}
\vskip1mm
\begin{proof}
 It is known from Lemma \ref{Lemma 5.3} that when the admissible ideal $I$ for $A$ is a homogeneous ideal, we have \[\text{Im}(\Phi_A)=\{M\in \text{GL}(J/J^2): \phi_M \in \text{Aut}_k(A) \text{ with } \phi_M(f+I)=f(XM)+I \}.\] Therefore, every $M\in\text{Im}(\Phi_A)$ induces an automorphism of $A$ which is given by the linear change of variables $F_1(X)=XM$, where $X=(X_1,\dots,X_n)$ and $F_1(I)\subset I$ by Proposition \ref{Proposition 5.2} and Lemma \ref{Lemma 5.3}. It is easy to see that $\bigcap_{i=1}^{m}\text{Sim}(f_i)\subset \text { Im}(\Phi_A)$. Let us denote $f_M(X)\coloneqq f(XM)$ and a linear change of variables by $F_1$ for $M$, where $M\in \text{ Im}(\Phi_A$). Since a linear change of variables does not change the degree of the polynomial and $F_1(f_i)={f_{i}}_{M}\in I$ by Lemma \ref{Lemma 5.3}, we have ${f_{i}}_M$ is a homogeneous polynomial of degree $d_i$ and ${f_{i}}_M=F_1(f_i)=\sum_{j=1}^{m}g_{ij}f_j$, where $g_{ij}\in k[X_{1},\dots,X_{n}]$ is a homogeneous polynomial of degree $s_{ij}$. As $F_1$ is given by the linear change of variables, we have
\begin{equation*}\label{eq30} 
d_i= \text{deg}({f_{i}}_{M})=\text{deg}(F_1(f_i))=\text{deg }(\sum_{j=1}^{m}g_{ij}f_j)=\text{ Max}_{j=1}^{m}(\text{deg}(g_{ij})+\text{deg}(f_j)).
 \end{equation*}
 Without loss of generality, let $d_1=\text{ min }d_i$, $d_i$ is degree of $f_i$ and $d_i
 \leq d_{i+1}$ for all $i\geq 2$. Since ${f_{i}}_M$ is a homogeneous polynomial of degree $d_i$, we have $d_i=(s_{ij}+d_j)$ for all $1\leq i,j\leq m$ only if $g_{ij}\neq 0$, where $\text{deg}(g_{ij})=s_{ij}$. By the given hypothesis, $f_1$ is the unique generator of minimum degree $d_1$. Therefore, $d_1=s_{11}+d_1$ and $d_1<d_i$ for all $i\neq 1$, which implies $s_{11}=0$ and $g_{1j}=0,\forall j\in\{2,\dots,m\}$. So, we get $f_1(XM)=\alpha f_1$. Hence, $\text{Im}(\Phi_A)\subset \text{Sim}(f_1)$. Finally, we have \[\bigcap_{i=1}^{m}\text{Sim}(f_i)\subset \text{Im}(\Phi_A)\subset \text{Sim}(f_1).\] Now from Proposition \ref{Proposition 5.4}, we know that $G_A\cong \text{Im}(\Phi_A)^{0}\ltimes \text{Ker}(\Phi_A)^{0}$. It is clear from Theorem \ref{Theorem 5.16} that under the assumption on $f_1$ depending on characteristic of $k$, $\text{Sim}^{0}(f_1)=\mathbb{G}_m$ as $\text{Stab}^{0}(f_1)=\{Id\}$. Hence, $\text{Im}(\Phi_A)^{0}\cong \mathbb{G}_m$. Therefore, $G_A\cong \mathbb{G}_m\ltimes \text{Ker}(\Phi_A)^{0}$, hence $G_A$ is a rational group of rank $1$. 
\end{proof}
\vskip2mm
\begin{lemma}\label{Lemma 5.18}
Let $A$ be a split local commutative graded by the radical algebra over $k$ with Loewy length $l$ and $\emph{dim}(J/J^2)=n$. Then the admissible ideal $I$ contains an $\emph{\text{Im}}(\Phi_A)$-stable finite-dimensional subspace $W$ with all non-zero elements having the same degree.
\end{lemma}
\vskip1mm
\begin{proof}
    Since $A$ is a graded by the radical local algebra, $I$ is isomorphic to $\langle X_1,\dots,X_n\rangle^l+\langle f_1,\dots,f_m\rangle$ with $f_i$' s  as homogeneous polynomials of $2\leq \text{deg}(f_i)\leq (l-1)$. Let $S$ denote the collection of polynomials of minimal degree in the generating set of $I$; \[S=\{f_i: \text{deg }f_i\leq \text{deg }f_j, \forall i\neq j \text{ and }i,j\in\{1,\dots,m\}\}.\] By renumbering if necessary, we can assume that $S=\{f_1,\dots,f_r\}$, where $r\leq m$. We claim that $W=\text{ Span}(S)$ is an $\text{Im}(\Phi_A)$-stable subspace under the action defined by $F_1.f=F_1(f)$, for all $F_1\in \text{Im}(\Phi_A)$ and $f\in k[X_1,\dots,X_n]$, where $F_1(f)(X)=f(XM_{F_1})$. We have $W=\{c_1f_1+\dots+c_rf_r:c_i\in k \}$. It is known that $F_1(I)\subset I$ for all $F_1\in \text{Im}(\Phi_A)$, by Lemma \ref{Lemma 5.3}, where $F_1$ is given by a linear change of variables. Let $0\neq g\in W$, then $F_1(g)\in I$. We have $\text{deg}(F_1(g))=\text{deg}(g)$ and $F_{1}(g)$ is a non-zero homogeneous polynomial as $F_1$ is given by a linear change of variables. Hence, $\text{deg}(g)=\text{deg}(F_{1}(g))=\text{Max}_{j=1}^m(\text{deg}(g_j)+\text{deg}(f_j))=\text{deg}(g_j)+\text{deg}(f_j)$, where $F_1(g)=\sum_{j=1}^{m}g_jf_j$, $g_j \text{ is a homogeneous polynomial in } k[X_1,\dots,X_n]$ and  $j\in\{1,\dots,m\}$. However, $\text{deg}(g)=\text{deg}(f_i)<\text{deg}(f_s)$, $\forall i\in \{1,\dots,r\},\forall s\in \{r+1,\dots,m\}$. Therefore, $\text{deg}(g_i)=0$ and $g_s=0$. This implies $F_1(g)=a_1f_1+a_2f_2+\dots+a_rf_r\in W$, where $a_i\in k.$
\end{proof}
\vskip2mm
\begin{remark}\label{Remark 5.19}
    The above subspace $W$ is independent of the choice of generators of $I$. We call this subspace the \emph{minimal degree subspace} associated to $A$ as it is generated by polynomials of minimal degree in $I$, and $I$ is determined by $A$.
\end{remark}
\vskip2mm
\begin{theorem}\label{Theorem 5.20}
Let $A$ be a split local commutative graded by the radical algebra over a perfect field $k$ with $\emph{dim}(J/J^2)=n$. Let $W$ be the minimal degree subspace associated to $A$ (see Remark \ref{Remark 5.19}). Let any of the conditions listed below be satisfied:
\begin{enumerate}
 \item $\emph{char }k=0$ and $W$ contains a non-singular homogeneous polynomial of degree at least $3$ with $\emph{\text{dim }}W\leq 3$;
\vskip2mm
\item $\emph{char }k=p>3$ and $W$ contains a non-singular homogeneous polynomial of degree $d$ with $2<d<p$ and $\emph{\text{dim }}W\leq 3$.
\end{enumerate}
 Then $G_A$ is $k$-rational.
\end{theorem}
\vskip2mm
\begin{proof}
Suppose $W$ satisfies the given hypothesis with $\text{dim }W\leq 3$. It is clear that, $\mathrm{Im}(\Phi_A)^{0}$ also preserves $W$. Let $f$ be a non-singular homogeneous polynomial in $W$ of degree at least $3$ if \text{char }k=0, and if $\text{char }k=p>3$, then $f$ is a non-singular homogeneous polynomial of degree $d$ with $2<d<p$. By orbit stabilizer theorem we have $\text{dim }(\mathrm{Im}(\Phi_A)^{0})-\text{dim } \mathrm{Stab}(f)=\text{dim }(O(f))$, where $O(f)$ is the orbit of $f$ under $\mathrm{Im}(\Phi_A)^{0}$. By Theorem \ref{Theorem 5.16}, we have $\text{dim }(\mathrm{Stab}(f))=0$ which implies $\text{dim }(\mathrm{Im}(\Phi_A)^{0})=\text{dim }(O(f))$. Again $\text{dim }(O(f))\leq \text{dim }W$ as $O(f)\subset W$. Hence $\text{dim }(\mathrm{Im}(\Phi_A)^{0})\leq \text{dim }W\leq 3$. By Proposition \ref{Proposition 5.5}, there exists a central $\mathbb{G}_m\subset \mathrm{Im}(\Phi_A)^0$, so we have $\mathrm{Im}(\Phi_A)^0\cong_b \mathrm{Im}(\Phi_A)^0/\mathbb{G}_m\times \mathbb{G}_m$ as $\mathrm{H}^1(F,\mathbb{G}_m)=\{1\}$ for every field extension $F$ over $k$. Again $\text{dim }(\mathrm{Im}(\Phi_A)^{0})/\mathbb{G}_m\leq 2$, which implies this group is $k$-rational. Hence $\mathrm{Im}(\Phi_A)^{0}$ is $k$-rational. Therefore, $G_A$ is $k$-rational by Proposition \ref{Proposition 5.4}.
\end{proof}

\section{Counter-Example to $R$-triviality of $G_A$}\label{counter}
\vskip2mm 

In this section, we construct a counter-example of Question \ref{Question 5.6} using the examples in the paper \cite{Nivedita}. The following are well-known examples of groups that are not $R$-trivial, and theorems associated with them. 
\vskip2mm
\noindent
Merkurjev in \cite{AM2} constructs a quadratic form $q$ of dimension $6$ with non-trivial signed discriminant over a base field $k$ of characteristic not $2$ and cohomological dimension $2$ such that $\text{PSO}(q)$ is not $R$-trivial as a variety, where $\text{PSO}(q)$ is the projective special orthogonal group.

\vskip2mm
\begin{theorem}\label{Theorem 6.1}
$($\cite{Gille2}$)$ There exists a field $k$ of characteristic $0$ with cohomological dimension $3$ and a quadratic form $q/k$ with rank 8 and trivial signed discriminant such that the variety $\emph{PSO}(q)$ is not stably $k$-rational.
\end{theorem}
\par 
\noindent
In addition, Ph.~Gille showed that the above group $\text{PSO}(q)$ is not $R$-trivial. Using Merkurjev's example, Nivedita Bhaskhar has produced more examples of adjoint groups which are not $R$-trivial in \cite{Nivedita}.
\vskip2mm

\begin{theorem}\label{Theorem 6.2}
$($\cite{Nivedita}, \emph{Theorem} $7.1)$ For each $n$, there exists a quadratic form $q_n$ (even dimension with $\emph{\text{dim}}(q_n)\geq 6$) defined over a field $k_n$ ($\text{char }k_n=0$) such that $\emph{PSO}(q_n)$ is not $R$-trivial as $k_n$-group.
\end{theorem}

\vskip2mm
\begin{theorem}\label{Theorem 6.3}
For each $n\geq 1$, there exists a finite-dimensional commutative algebra $A_n$ with $\emph{dim} \displaystyle (\frac{J_n}{J_n^2})\geq 6$, where $J_n$ is the Jacobson radical of $A_n$ over a field $k_n$, such that $G_{A_n}$ is not $R$-trivial as a $k_n$-group. 
\end{theorem}
\vskip1mm
\begin{proof}
Let $q$ be an even dimensional quadratic form over $k$, then we have \[\mathrm{Sim}^{0}(q)\cong_b \mathrm{Sim}^{0}(q)/\mathbb{G}_m\times \mathbb{G}_m\] as $H^1(F,\mathbb{G}_m)=\{1\}$ for every field extension $F$ over $k$. Therefore, we have $\mathrm{Sim}^0(q)(L)/R=\mathrm{PSO}(q)(L)/R$ for every extension $L$ over $k$ as \[\mathrm{Sim}^{0}(q)/\mathbb{G}_m=\mathbb{G}_m.\mathrm{Stab}^{0}(q)/\mathbb{G}_m=\mathrm{SO}(q)/\mathrm{Z(SO(q))}=\mathrm{PSO}(q).\] For each $n\geq 1$ by Theorem \ref{Theorem 6.2}, there exists a quadratic form $q_{n}$ defined over a field $k_{n}$ such that $\mathrm{PSO}(q_{n})$ is not $R$-trivial as $k_n$-group with $\text{dim}(q_n)\geq 6$. Now consider the algebra $A_n=k_n[X_1,\dots,X_{m}]/I_n$, where $I_n=\langle X_1,\dots,X_{m}\rangle ^{3}+\langle q_{n}\rangle$ with $m=\text{dim}(q_n)\geq 6$. Hence, we have \[G_{A_n}(k_n)/R=\mathrm{Sim}^{0}(q_{n})(k_n)/R=\mathrm{PSO}(q_{n})(k_n)/R\neq (1)\] by Theorem \ref{Theorem 5.12} and Theorem \ref{Theorem 6.2} with $\text{dim }(J_n/J_n^2)=m\geq 6$. 
\end{proof}
\vspace{2mm}
\begin{remark}\label{Remark 6.4}
Let $G\subset \mathrm{SL}(V)$ be a connected simple adjoint group over a field $k$ of \text{char }k=0, with $ V$ being an irreducible representation of $ G$. Suppose $G$ is not $R$-trivial, and there exists a homogeneous polynomial $f\in k[V]$ of degree at least $3$ such that $G=\mathrm{Stab}_{\mathrm{GL}(V)}^{0}(f)$ (see Lemma $5.1$ of \cite{GG}); then we can construct more commutative algebras $A$ such that $G_A$ is not $R$-trivial by applying a similar technique as discussed above.
\end{remark}

 \noindent
 {\bf Acknowledgements:} I thank the Indian Statistical Institute Bangalore for providing an excellent, productive research environment during this work. \\
\begin{center}
{\bf Declarations}
\end{center}
\vskip2mm
\noindent
{\bf Data Availability:} No data sets were generated or analyzed during this work.\\
\noindent
{\bf Conflict of interest:} The author states that there is no conflict of interest. \\
\noindent
{\bf Funding Declaration:} There was no external funding for this research work.
\bibliography{references}
\bibliographystyle{amsplain}
\end{document}